\documentclass[11pt]{amsart}
\usepackage{amsmath,amsthm,amscd,amsfonts,amssymb,graphicx,color,MnSymbol,enumerate,tikz-cd}
\usepackage{hyperref,cleveref}
\usepackage{xcolor}
\usepackage[labelformat=empty]{caption,subcaption}
\usepackage{inputenc}
\usepackage{stmaryrd}

\newtheorem{theorem}{Theorem}[section]
\newtheorem{lemma}[theorem]{Lemma}
\newtheorem{proposition}[theorem]{Proposition}
\newtheorem{corollary}[theorem]{Corollary}
\theoremstyle{definition}
\newtheorem{definition}[theorem]{Definition}

\theoremstyle{remark}
\newtheorem{remark}[theorem]{Remark}

\numberwithin{equation}{section}

\allowdisplaybreaks

\begin{document}
	
	\title[On Regular Mac Lane-Vaqui\'e chains]{On Regular Mac Lane-Vaqui\'e chains}
	\author[Nikita Dwivedi]{Nikita Dwivedi}
	\address{Department of Mathematics\\ University of Delhi\\  Delhi-110007, India.}
	\email{ndwivedi@maths.du.ac.in}
	\author[Anuj Bishnoi]{Anuj Bishnoi$^\ast$}
	\address{Department of Mathematics\\  University of Delhi \\   Delhi-110007, India.}
	\email{abishnoi@maths.du.ac.in}
	
	\begin{abstract}
   Let $(K^h,v^h)$ be a Henselization of a valued field $(K,v)$, $\mu$ any extension of $v$ to $K[x]$, and $\mu^h$ the canonical extension of $\mu$ and $v^h$ to $K^h[x].$ In this paper, we study two new invariants associated with a Mac Lane-Vaqui\'e chain (MLV) of $\mu$, namely, the defect of each augmentation step and the regularity of an MLV chain of $\mu$. This, in turn, extends the notion of regularity of a complete sequence of abstract key polynomials for $\mu$ to the regularity of $\mu$ itself. We also prove that the regularity of $\mu$ is a sufficient condition for the equality of the depths of $\mu$ and $\mu^h$. Furthermore, it is a necessary and sufficient condition for the equality of the defect and relative gap of each augmentation step of an MLV chain of $\mu.$ Finally, we show that if $\mu$ has finite degree, then the regularity of $\mu$ is equivalent to $\deg(\mu)=\deg(\mu^h).$
	\end{abstract}
	\thanks{$^\ast$Corresponding author, E-mail address: abishnoi@maths.du.ac.in}
	\maketitle
	
	
	\section{Introduction}
Let $(K,v)$ be a valued field, $\bar{v}$ an extension of $v$ to the fixed algebraic closure $\overline{K}$ of $K,$ and $K^h$ the Henselization of $K$ with respect to $\bar{v}.$ Let $\mu$ be an extension of $v$ to $K[x].$

In 2021, Nart \cite{EN1} introduced Mac Lane-Vaquie (MLV) chains of $\mu$, as a mixture of ordinary and limit augmentations of valuations on $K[x]$, extending $v$:
\begin{align*}
	(v\xrightarrow{\phi_0,\gamma_0})\mu_0\xrightarrow{\phi_1,\gamma_1} \mu_1\xrightarrow{\phi_2,\gamma_2}\cdots \longrightarrow \mu_{n}\xrightarrow{\phi_{n+1},\gamma_{n+1}} \mu_{n+1}\longrightarrow\cdots\longrightarrow \mu
\end{align*}
 and proved that every valuation on $K[x]$ is a limit of a countable MLV chain. This chain is unique with respect to intrinsic data such as depth, relative gaps, degrees of the nodes and the nature of each augmentation step. The defect of an augmentation was defined by Nart and Novacoski in \cite{EN3}. In this paper, we prove that the defect of an augmentation step in an MLV chain of $\mu$ is also an intrinsic characteristic, i.e., it does not depend on the choice of an MLV chain of $\mu$ (see Theorem \ref{td1}).

Let $\overline{\mu}$ be a common extension of $\mu$ and $\bar{v}$ to $\overline{K}[x],$ and let $\mu^h$ be its restriction to $K^h[x].$ In 2004, Kuhlmann  proved that both $\mu$ and $\overline{\mu}$ fall into the same category, viz., valuation-transcendental, valuation-algebraic or valuation with nontrivial support (see \cite[Theorem 3.11]{K}). In fact, $\mu^h$ also falls into the same category. Therefore, MLV chains of $\mu$ and $\mu^h$ are also of the same Type as in Theorem \ref{mlvt1}. However, the depth of an MLV chain of $\mu$ and $\mu^h$ need not be the same (see Examples 8.4 and 8.6 of \cite{EN3}). In this paper, we prove that not only their depths but also their defining pairs are the same, under the 
\begin{quote}
{\bf Hypothesis}: {\it each defining key polynomial $\phi_n$ of an MLV chain of $\mu$ is irreducible over $K^h,$ i.e., $\phi_n=\phi_n^h.$}	
\end{quote}
\smallskip

In 2007, Vaqui\'e proved that if a finite simple extension $K(\theta)|K$ is unibranched, then its Henselian defect equals the product of the relative gaps of the augmentation steps in an MLV chain of a valuation $\mu,$ induced by the extension of $v$ to $K(\theta)$ (see Theorem \ref{vth1}). Later, Nart and Novacoski gave another characterization of defect of an arbitrary finite simple extension, expressing it in terms of the product of the defects of the augmentation steps (see Theorem \ref{T2}). They also proved that if $(K,v)$ is Henselian, then the defect and relative gap of any augmentation are equal (see Lemma \ref{nl2}). In Theorem \ref{mt2}, we generalize this result to arbitrary valued fields, for the augmentation steps of an MLV chain of $\mu$  satisfying the hypothesis. In particular, by Theorems \ref{mt2} and \ref{lst}, it follows that if $\mu$ is induced by an extension of $v$ to $K(\theta)$, then the defect and relative gap of each augmentation step of an MLV chain of $\mu,$ are equal if and only if $K(\theta)|K$ is unibranched.

Recently, Dutta and Mahboub \cite{AM} introduced the notion of regularity for a complete sequence of abstract key polynomials (ABKPs) for $\mu$. Using the well-known connection between an MLV chain of $\mu$ and a complete sequence of ABKPs for $\mu$, we show that this notion of regularity is equivalent to the hypothesis. Keeping this in mind, we call any MLV chain of $\mu$ satisfying the hypothesis  a regular MLV chain. As a consequence of Theorem \ref{mt2}, we also obtain that the regularity is independent of the choice of an MLV chain of $\mu$ (see Corollary \ref{lc}). In particular, if a complete sequence of ABKPs for $\mu$ is regular, then all complete sequences of ABKPs are also regular. Thus, we say that $\mu$ is regular if any MLV chain of $\mu$ is regular. Moreover, for a finite degree valuation $\mu$, the hypothesis is equivalent to the condition that $\deg(\mu)=\deg(\mu^h)$ (see Theorem \ref{lst}). Finally, we show that if $\mu$ is regular, then any valuation less than $\mu$ is also regular. 
	 
	\section{Preliminaries}
		 In this section, we recall some notation, definitions, and results that will be used in the proofs of our main results.
		
 A surjective map $v:K \longrightarrow\ \Gamma_v\infty:=\Gamma_v\cup\{\infty\},$ where $\Gamma_v$ is a totally ordered additively written abelian group, is called a {\bf valuation} if it satisfies the following axioms, for all $a,b$ in $K$:
	\begin{enumerate}[(i)]
		
\item $v(ab)=v(a)+v(b)$
	\item $v(a+b)\geq \min\{v(a),v(b)\}$
			\item $v(a)=\infty$ if and only if $a=0$.
	\end{enumerate}
		The pair $(K,v)$ is called a {\bf valued field} of arbitrary rank, and  $\Gamma_ v$  its {\bf value group}. The set $O_v=\{a\in K \mid v(a)\geq 0\}$ is a subring of $K,$ called the {\bf valuation ring}, which has a unique {\bf maximal ideal} $M_v=\{a\in K\mid v(a)>0 \}.$ The quotient $O_v/M_v$ is called the {\bf residue field} of $v$ and is denoted by $k_v.$  Let $\bar{v}$ be an extension of $v$ to a fixed algebraic closure $\overline{K}$ of $K$ with value group $\Gamma_{\bar{v}}.$ 
		
 Let $\mu:K[x]\longrightarrow\Gamma\infty$ be an extension of $v$ satisfying axioms (i) and (ii) above, where $\Gamma$ is some ordered abelian group containing $\Gamma_{\bar{v}}.$ The {\bf support}  of $\mu$ is the prime ideal
 \[\text{supp}(\mu)=\mu^{-1}(\infty).\] If $\text{supp}(\mu)=\{0\},$ then it extends uniquely to a valuation on $K(x).$ If $\text{supp}(\mu)\neq\{0\},$ then we say that $\mu$ has a nontrivial support. In either case, we call $\mu$ a valuation on $K[x],$ and define the value group $\Gamma_\mu$ of $\mu$ as the group generated by $\mu(K[x]\backslash\text{supp}(\mu)),$ and the residue field $k_\mu$ of $\mu$ as the residue field of the canonical valuation induced by $\mu$ (see Remark \ref{r21}) on the quotient field of $K[x]/\text{supp}(\mu)$.

 An extension $\mu$ of $v$ to $K(x)$ satisfies the well-known Abhyankar inequality: \[\text{rr}(\Gamma_\mu / \Gamma_v) + \text{tr.deg.}[k_{\mu}: k_v] \leq 1,\] where $\text{rr}(\Gamma_\mu / \Gamma_v )$ is the rational rank of $\Gamma_\mu / \Gamma_ v,$ and $\text{tr.deg.}[k_{\mu}: k_v]$ is the transcendence degree of $k_{\mu}$ over $k_v.$ The extension $\mu$ is said to be {\bf value-transcendental} if $\text{rr}(\Gamma_\mu / \Gamma_v)=1$ and is said to be {\bf residue-transcendental} if $\text{tr.deg.}[k_{\mu}: k_v]=1$. We call $\mu$ {\bf valuation-transcendental} if it is either value-transcendental or residue-transcendental. Otherwise, $\mu$ is called {\bf valuation-algebraic}.

An extension $\overline{\mu}$ of $\mu$ to $\overline{K}[x]$ which is also an extension of $\bar{v}$ is called a {\bf common extension} of $\mu$ and $\bar{v}$. 

For any pair $(\alpha,\delta)\in\overline{K}\times\Gamma\infty,$ the map $\overline{w}_{\alpha,\delta}: \overline{K}[x]\longrightarrow \Gamma\infty,$ given by
	\begin{align*}
		\overline{w}_{\alpha,\delta}\left(\sum_{i\geq 0} c_i (x-\alpha)^i\right):=\min_{i\geq 0}\{\bar{v}(c_i)+i\delta\}, \, c_i\in\overline{K},
	\end{align*}
	is a valuation on $\overline{K}[x]$ and is said to be defined by $\min,\, \bar{v},\, \alpha$ and $\delta.$ If $\delta=\infty,$ then it has a nontrivial support generated by $(x-\alpha).$
		Let $\overline{\mu}$ be a common extension of $\mu$ and $\bar{v}$ to $\overline{K}[x]$ such that $\overline{\mu}=\overline{w}_{\alpha,\delta},$ then $(\alpha,\delta)$ is called a {\bf pair of definition} for $\overline{\mu}.$ We denote its restriction $\mu$ to $K[x]$ by $w_{\alpha,\delta}.$
	\begin{lemma}(\cite[Lemma 2.4]{W})\label{il1}
		Let $(\alpha,\delta)$ be a pair of definition for $\overline{\mu}$ and $(\alpha',\delta')\in \overline{K}\times\Gamma.$ Then $(\alpha',\delta')$ is also a pair of definition for $\overline{\mu}$ if and only if $\delta'=\delta$ and $\bar{v}(\alpha-\alpha')\geq \delta.$
	\end{lemma}
	
	\begin{definition}
	A pair $(\alpha,\delta)$ in $\overline{K}\times \Gamma$ is called  a $(K,v)$-{\bf minimal pair} if for every $\beta$ in $\overline{K}$ satisfying $\bar{v}(\alpha-\beta)\geq\delta,$ we have  $\deg\beta\geq\deg\alpha,$ where by $\deg\alpha$ we mean the degree of the extension $K(\alpha)|K.$  
\end{definition}

 If $\overline{\mu}=\overline{w}_{\alpha,\delta},$ such that $(\alpha,\delta)$ is a $(K,v)$-minimal pair, then we say that $(\alpha,\delta)$ is a {\bf minimal pair of definition} for $\overline{\mu}.$ In view of Lemma \ref{il1}, any pair of definition for $\overline{\mu}$ can be replaced by a minimal pair of definition.

Moreover, for any $(\alpha,\delta), (\alpha',\delta')\in \overline{K}\times\Gamma\infty,$ we have 
\begin{align}\label{eb}
\overline{w}_{\alpha,\delta}(f)\leq \overline{w}_{\alpha',\delta'}(f),\ \text{for all} \ f\in\overline{K}[x]\iff \delta\leq \delta' \ \text{and} \ \bar{v}(\alpha-\alpha')\geq \delta.	
\end{align}

If  $\delta<\infty,$ then $w_{\alpha,\delta}$ is a valuation-transcendental extension, and any valuation-transcendental extension can be obtained in this way \cite[Corollary 3.7]{W}. The valuation $w_{\alpha,\delta}$ is residue-transcendental if $\delta\in\Gamma_{\overline{v}}$ and value-transcendental if $\delta\notin\Gamma_{\overline{v}}$. On the other hand, if $\delta=\infty,$ then $w_{\alpha,\delta}$ has a nontrivial support generated by the minimal polynomial of $\alpha$ over $K.$ If $\alpha\in K,$ then the valuation $w_{\alpha,\delta}$ is called a {\bf depth zero valuation}.

\subsection{Key Polynomials}

We now recall the definition of key polynomials, first introduced by Mac Lane \cite{M} in 1936 for discrete rank-one valuations and later generalized by Vaqui\'e \cite{V} in 2007 for arbitrary valuations. Further, Nart \cite{EN1} classified all possible extensions of $v$ to $K[x]$ using key polynomials.

For a valuation-transcendental extension  $\mu$ of $v$ to $K(x)$ and for polynomials $f,g$ in $K[x],$ we say that $f$ and $g$ are {\bf $\mu$-equivalent} (denoted $f\sim_\mu g$)  if $\mu(f-g)>\mu(f)=\mu(g)$ and $g$ is {\bf $\mu$-divisible} by $f$ (denoted $f\mid_{\mu} g$)  if there exists a polynomial $h \in K[x]$ such that $g$ is $\mu$-equivalent to $fh.$

		\begin{definition}
		A monic polynomial $\phi\in K[x]$   is called a  {\bf  key polynomial} for $\mu$ if it is
		\begin{enumerate}[(i)]
			\item  $\mu$-{\bf irreducible}, i.e., for any $f,\, g\in K[x],$ whenever $\phi\mid_{\mu} fg,$ then either $\phi\mid_{\mu}f $ or $\phi\mid_{\mu}g,$ and 
			\item $\mu$-{\bf minimal}, i.e., for every nonzero polynomial $f\in K[x],$ whenever $\phi\mid_{\mu}f,$ then $\deg f\geq \deg \phi.$
		\end{enumerate}
	\end{definition}
	
	The set of all key polynomials for $\mu$ is denoted by $\operatorname{KP}(\mu).$ If $\mu=w_{\alpha,\delta}$ for some $(K,v)$-minimal pair $(\alpha,\delta)\in \overline{K}\times\Gamma$ and $\phi$  is the minimal polynomial of $\alpha$ over $K,$ then $\phi$ is a key polynomial for $\mu$ of minimal degree (\cite[Theorem 1.1]{JN} and \cite[Theorem 2.21]{Ma}). 
The existence of key polynomials is characterized as follows.
\begin{theorem}(\cite[Theorem 4.4]{EN4})
	A valuation $\mu$ on $K(x)$ has $\operatorname{KP}(\mu)\neq\emptyset$ if and only if it is valuation-transcendental.
\end{theorem}

		Let $\phi\in K[x]$ be  any key polynomial for $\mu$. Then, by Proposition 2.3 of \cite{EN4}, for  any polynomial $f\in K[x]$ with $\phi$-expansion $\sum_{i=0}^{n}   f_i \phi^i,$ $f_i\in K[x],$  $\deg f_i<\deg \phi,$ we have 
	\begin{align*}
		\mu(f)=\min_{0 \leq i \leq n}\{\mu(f_i)+ i\mu(\phi)\},
	\end{align*}
	and denote the set $S_{\mu,\phi}(f):=\{0\leq i\leq n\mid \mu(f)=\mu(f_i\phi^i)\}.$
	If $\phi$ has minimal degree in $\operatorname{KP}(\mu),$ then for any nonzero $f\in K[x],$ we denote   
	\[\deg_\mu(f):=\max(S_{\mu,\phi}(f)).\] Clearly, this is independent of the choice of $\phi$ among all minimal degree key polynomials for $\mu.$
	 
	For  valuations $\mu$ and $\nu$ on $K[x],$ taking values in $\Gamma,$ we say that $\mu\leq \nu$ if and only if 
	$$\mu(f)\leq \nu(f),~\forall f\in K[x].$$
	\begin{definition}
		Let $\mu<\nu.$ The set of all monic polynomials $g\in K[x]$ of minimal degree (say) $d'$ such that $\mu(g)<\nu(g),$ denoted  by $\Phi(\mu,\nu),$ is called the {\bf tangent direction} of $\mu,$ and  $\deg(\Phi(\mu,\nu)):=d'.$ 
\end{definition}

 The set $\Phi(\mu,\nu)\subset \operatorname{KP}(\mu)$ (see \cite[Theorem 1.15]{V}), which implies that $\operatorname{KP}(\mu)\neq\emptyset$, whenever $\mu<\nu$ for some valuation $\nu$ on $K[x].$
 
 \begin{corollary}(\cite[Corollary 2.5]{EN1})\label{2.5}
 	Let $\mu<\nu$ be as above, and let $\rho$ be another valuation. Then
 	\begin{enumerate}[(i)]
 		\item $\Phi(\mu,\nu)=[\phi]_\mu=\{\psi\in K[x]\mid \psi\sim_{\mu}\phi\},$ for all $\phi\in\Phi(\mu,\nu),$
 		\item If $\mu<\rho<\nu$, then $\Phi(\mu,\nu)=\Phi(\mu,\rho)$.  In particular, $\mu(f)=\nu(f)$ if and only if $\mu(f)=\rho(f),$ for all $f\in K[x].$
 	\end{enumerate}
 \end{corollary}
 
 For any valuation $\mu$ on $K[x]$, we now define the degree of $\mu,$ denoted by $\deg(\mu)$, as follows:
 \begin{itemize}
 	\item If $\operatorname{KP}(\mu)\neq\emptyset,$ then we define $\deg (\mu):=\min_{\phi\in\operatorname{KP}(\mu)}\deg \phi.$
 	\item  If $\text{supp}(\mu)=\phi K[x]$ for some $\phi\in K[x],$ then we define $\deg (\mu):=\deg\phi.$
 	\item  If $\operatorname{KP}(\mu)=\emptyset$ and $\text{supp}(\mu)=\{0\}$, i.e., $\mu$ is valuation-algebraic, then we define
 	\begin{align*}
 		\deg(\mu):=\sup(\text{Deg}((-\infty,\mu)_{\Gamma}))\in\mathbb{N}\infty,
 	\end{align*} where 
 	\begin{align}
 		(-\infty,\mu)_{\Gamma}=\{\rho:K[x]\longrightarrow \Gamma\infty\mid \rho\ \text{is valuation},\rho\mid_K=v,\rho<\mu\},
 	\end{align} and for any open interval $I$ of valuations, $\text{Deg}(I)=\{\deg(\rho)\mid  \rho\in I\}$ (see \cite[Section 2C]{EN1}). 
 	\end{itemize}

\subsection{Mac Lane-Vaqui\'e chains}

We now define augmentations and Mac Lane-Vaqui\'e (MLV) chains of a valuation on $K[x].$
\begin{definition}
	Let $\phi$ be a key polynomial for a valuation $\mu$ and $\gamma\in\Gamma\infty$ such that $\gamma> \mu(\phi)$. The map $\nu: K[x]\longrightarrow \Gamma\infty$ defined by 
	$$\nu(f):=\min_{i\geq 0}\{\mu(f_i)+i\gamma\},$$ where
	$\sum_{i\geq 0}f_i \phi^i, $ $\deg f_i<\deg \phi,$ is the $\phi$-expansion of $f\in K[x],$  gives a valuation on $K[x]$  called the {\bf ordinary augmentation of} $\mu,$ defined by $\phi$ and $\gamma,$ and is  denoted  by $[\mu; \phi,\gamma].$
\end{definition}
Note that $\nu(\phi)=\gamma,$ i.e., $\mu(\phi)<\nu(\phi).$ If $\gamma<\infty,$ then $\phi\in\operatorname{KP}(\nu)$ of minimal degree (\cite[Proposition 2.1]{EN1}), otherwise, $\text{supp}(\nu)=\phi K[x].$ In either case, $\deg(\Phi(\mu,\nu))=\deg (\nu)=\deg\phi.$ 

Let $\mathbf{A}$ be a well-ordered set without a last element. A family $\mathcal{C}=(\rho_i)_{i\in\mathbf{A}}$ of valuations on $K[x],$ taking values in $\Gamma$, is said to be a {\bf continuous family}, parametrized by $\mathbf{A}$, if:
\begin{itemize}
\item The map $i\longrightarrow \rho_i$ is an isomorphism of totally ordered sets.
\item The set $\{\deg(\rho_i)\}_{i\in\mathbf{A}}$ is stable, i.e., there exists $i_0\in\mathbf{A}$ such that $\deg(\rho_i)=\deg(\rho_{i_0})$ for all $i\geq i_0$. We denote this stable degree by $\deg(\mathcal{C})$.
\end{itemize}
A polynomial $f$ in $K[x]$ is said to be {\bf $\mathcal{C}$-stable}  if there exists some index $i_0\in\mathbf{A}$ such that
$$\rho_i(f)=\rho_{i_0}(f),~ \forall~ i\geq i_0,$$  and this stable value is denoted by $\rho_\mathcal{C}(f).$ Otherwise, $f$ is said to be {\bf $\mathcal{C}$-unstable}. By Corollary \ref{2.5} (ii), we have $f\in K[x]$  is $\mathcal{C}$-unstable if and only if  
$$\rho_i(f)<\rho_j(f),\hspace{5pt} \forall~ i<j\in\mathbf{A}.$$  
We denote 
$$m_{\infty}:=\min\{\deg f\mid f\in K[x],~\text{$f$ is $\mathcal{C}$-unstable}\}.$$ If  all polynomials are $\mathcal{C}$-stable, then we say that $\mathcal{C}$ has a {\bf stable limit} and set $m_{\infty}=\infty$. In this case, $\rho_{\mathcal{C}}$ is a valuation on $K[x],$ called the {\bf stable limit} of $\mathcal{C}$ (see \cite[Proposition 3.1]{EN1}). 

\begin{definition}
	A monic polynomial $\phi\in K[x]$ is said to be a {\bf limit key polynomial} for the family $\mathcal{C}$ if it is $\mathcal{C}$-unstable and $\deg\phi=m_{\infty}<\infty.$
\end{definition}

The set of all limit key polynomials for $\mathcal{C}$ is denoted by $\operatorname{KP}_{\infty}(\mathcal{C}).$  Since the product of stable polynomials is stable, all limit key polynomials are irreducible in $K[x].$

%
%
%
\begin{definition}
	Let   $\phi$ be any limit key polynomial for a continuous family $\mathcal{C}=(\rho_i)_{i\in\mathbf{A}}$ of valuations on $K[x]$ and  $\gamma\in \Gamma\infty$ such that $\gamma>\rho_i(\phi)$ for all $i\in\mathbf{A}.$ Then the map $\nu: K[x]\longrightarrow\Gamma\infty$ defined by
	$$\nu(f):=\min_{i\geq 0}\{\rho_\mathcal{C}(f_i)+i\gamma\},$$
	where  $\sum_{i\geq 0} f_i\phi^i,$ $\deg f_i<\deg \phi,$ is the $\phi$-expansion of $f\in K[x],$ gives a valuation on $K[x]$ and is called  the {\bf limit augmentation} of $\mathcal{C},$  denoted by $\nu=[\mathcal{C}; \phi, \gamma].$ 
\end{definition}
  Note that $\nu(\phi)=\gamma$ and $\rho_i<\nu$ for all $i\in\mathbf{A}.$ Also if $\gamma<\infty$, then $\phi$ is a key polynomial for $\nu$ of minimal degree  \cite[Proposition 3.5]{EN1}, otherwise, $\text{supp}(\nu)=\phi K[x].$ In either case, $\deg (\nu)=\deg\phi.$

\begin{definition}
	Let $\mathcal{C}=(\rho_i)_{i\in\mathbf{A}}$ be a continuous family of valuations on $K[x]$ and let $\phi\in\operatorname{KP}_{\infty}(\mathcal{C})$. Consider the supremum \[\gamma_{\mathcal{C}}=\sup\{\rho_i({\phi})\mid i\in\mathbf{A}\}.\]
			The limit augmentation $\nu_{\mathcal{C}}=[\mathcal{C};\phi,\gamma_{\mathcal{C}}],$ taking values in $\mathbb{R}_{sme}$, an ordered group containing $\Gamma$ (for the definition and existence of $\mathbb{R}_{sme}$, we refer the reader to \cite[Section 5.3]{EN3}), is called the {\bf minimal limit augmentation} of $\mathcal{C}$. By \cite[Lemma 4.7]{Ma2}, the value $\gamma_{\mathcal{C}}$ and the valuation $\nu_{\mathcal{C}}$ do not depend on the choice of the limit key polynomial  $\phi.$ 
\end{definition}

\begin{lemma}(\cite[Section 7.3.3]{Ma2})\label{vt1}
Let $\nu_{\mathcal{C}}$ be the minimal limit augmentation of a continuous family $\mathcal{C}$. Then, $\nu_{\mathcal{C}}$ is value-transcendental and $\operatorname{KP}(\nu_{\mathcal{C}})=\operatorname{KP}_{\infty}(\mathcal{C}).$ Moreover, every limit augmentation $\nu=[\mathcal{C};\phi,\gamma],$ for $\gamma>\gamma_{\mathcal{C}},$ can be obtained as the ordinary augmentation $\nu=[\nu_{\mathcal{C}};\phi,\gamma].$ 
\end{lemma}

\begin{definition}\label{1.1.14}
	Let $\mu$ be a valuation on $K[x]$ admitting key polynomials. Then a {\bf continuous family of augmentations} of $\mu$ is a continuous family $\mathcal{C} =(\rho_i=[\mu;\chi_i,\beta_i])_{i\in\mathbf{A}},$ of ordinary augmentations of $\mu$,  indexed by a well-ordered set $\mathbf{A}$ without a last element.
	\end{definition}

Any continuous family $\mathcal{C}$ of augmentations of $\mu$ on $K[x]$ falls in one of the following three cases:
\begin{enumerate}[(i)]
	\item  It has a stable limit, if $m_{\infty}=\infty.$ 
	\item It is {\bf in-essential} if $\deg(\mathcal{C})=m_{\infty}<\infty.$
	\item It is  {\bf essential }  if $\deg(\mathcal{C})<m_{\infty}<\infty.$
\end{enumerate} 

\begin{remark}\label{cfr}
	 Arguing as in Lemma 4.11 of \cite{Ma2}, any essential continuous family of augmentations of $\mu$ can be replaced by an equivalent\footnote{ Any two continuous families are said to be equivalent if they are cofinal in each other.} family $\mathcal{C}=(\rho_i=[\mu;\chi_i,\beta_i])_{i\in\mathbf{A}}$ of augmentations of $\mu,$ that satisfies the following extra conditions:
\begin{enumerate}[(i)]
	\item For all $i$ in $\mathbf{A},$ $\chi_i\in \operatorname{KP}(\mu)$ have the same degree.
	\item For all $i<j$ in $\mathbf{A},$ $\beta_i<\beta_j.$
	\item For all $i<j$ in $\mathbf{A},$ $\chi_j\in\operatorname{KP}(\rho_i),$ 
	$\chi_j\not\sim_{\rho_i}\chi_i~\text{and}~ \rho_j=[\rho_i;\chi_j,\beta_j].$
\end{enumerate}	
\end{remark}
	The properties given below directly follow from the remark above.
	\begin{enumerate}[(i)]
		\item The mapping defined by $i\mapsto\beta_i$ and $i\mapsto\rho_i$ are isomorphisms of ordered sets  $\mathbf{A}$ and $\{\beta_i\mid i\in\mathbf{A}\},$ $\{\rho_i\mid i\in\mathbf{A}\},$ respectively.
		\item For all $i, j\in\mathbf{A},$ $\rho_i(\chi_j)=\min\{\beta_i,\beta_j\},$ i.e., all $\chi_i$ are $\mathcal{C}$-stable.
		\item $\Phi(\rho_i,\rho_j)=[\chi_j]_{\rho_i},$ $\forall$ $i<j\in\mathbf{A}.$
		\item All valuations $\rho_i$ are residue-transcendental.
		\item All the value groups $\Gamma_{\rho_i}$ coincide and  the common value group is denoted by $\Gamma_{\mathcal{C}}.$ 
	\end{enumerate}


The following result establishes the existence of continuous families of augmentations of $\mu$ of the above type, whenever $\mu$ is valuation-transcendental.
\begin{proposition}(\cite[Proposition 3.3]{EN1})\label{cf1}
	Let $\nu$ be a valuation on $K[x]$ such that $\mu<\nu.$ Suppose that the set $\mathbb{A}=\nu(\Phi(\mu,\nu))$ does not contain a maximal element in $\Gamma_{\nu}\infty$. For all $\alpha\in\mathbb{A},$ choose some polynomial $\chi_{\alpha}\in\Phi(\mu,\nu)$ such that $\nu(\chi_{\alpha})=\alpha,$ and let $\rho_\alpha=[\mu;\chi_\alpha,\alpha].$ Then, $\mathcal{A}=(\rho_{\alpha})_{\alpha\in\mathbb{A}}$ is a continuous family of augmentations of $\mu$ satisfying the properties of Remark \ref{cfr}. 
\end{proposition}

\begin{remark}\label{r1}
The family $\mathcal{A}$ above is independent of the choice of $\chi_\alpha$. Indeed, if we choose another polynomial $\chi\in \Phi(\mu,\nu)$ with $\nu(\chi)=\alpha,$ then for $a=\chi-\chi_\alpha,$ we have $\nu(a)\geq \min\{\nu(\chi),\nu(\chi_\alpha)\}=\alpha.$ Also, since $\deg a<\deg\chi=\deg(\Phi(\mu,\nu)),$ we have $\mu(a)=\nu(a),$ which implies that $\mu(\chi-\chi_\alpha)\geq \alpha.$ Therefore, by the uniqueness of ordinary augmentations (see \cite[Lemma 2.8]{EN1}), we have $\rho_\alpha=[\mu;\chi_\alpha,\alpha]=[\mu;\chi,\alpha].$	
\end{remark}

	If $\nu=[\mathcal{C};\phi,\gamma]$ is a limit augmentation of a continuous family $\mathcal{C}$ of augmentations of $\mu$, then we say that $\mu\longrightarrow \nu$ is a limit augmentation.

\begin{lemma}
	Let $\mu\longrightarrow\nu$ be a limit augmentation such that $\nu=[\mathcal{C};\phi,\gamma]$, where $\mathcal{C}=(\rho_i)_{i\in\mathbf{A}}$ is an essential continuous family of augmentations of $\mu.$ Then $\nu(\Phi(\mu,\nu))$ does not contain a maximal element.
\end{lemma}
\begin{proof}
	Assume to the contrary that $\mathbb{A}=\nu(\Phi(\mu,\nu))$ contains a maximal element (say) $\alpha=\nu(\chi),$ for some $\chi\in\Phi(\mu,\nu).$ Define an ordinary augmentation $\eta=[\mu;\chi,\alpha].$ For any $i\in\mathbf{A}$, since $\rho_i=[\mu;\chi_i,\beta_i]$, we have $\Phi(\mu,\rho_i)=\Phi(\mu,\nu)=[\chi_i]_{\mu},$ i.e., all key polynomials $\chi_i$ are $\mu$-equivalent. Consequently, the value $\beta_{min}:=\mu(\chi_i)$ is independent of $i\in\mathbf{A}$. We have \[\deg(\eta)=\deg(\Phi(\mu,\nu))=\deg(\mathcal{C})<\deg\phi=\deg(\nu).\] Hence $\mu<\eta<\nu,$ i.e., $\eta\in(\mu,\nu)_{\Gamma}$ with $\deg(\eta)=\deg(\mathcal{C})<\deg\phi.$ 
	
	If $S$ is the initial segment of $\Gamma_{>\beta_{min}}=\{\gamma\in\Gamma\mid\gamma>\beta_{\min}\}$ generated by the set $\{\beta_i\mid i\in\mathbf{A}\},$ then by Lemma 3.8 of \cite{EN1}, $\eta=\rho_\beta=[\mu;\chi_i,\beta]$ for some $\beta\in S$ and $i\in\mathbf{A}$ with $\beta_i>\beta.$ By the uniqueness of ordinary augmentation (see \cite[Lemma 2.8]{EN1}), we have $\alpha=\beta$, which implies $\alpha<\beta_i$, contradicting the maximality of $\alpha.$ 
	\end{proof}
	
	\begin{remark}\label{lmr}
		With notation and hypothesis as in the above lemma, the family $\mathcal{A},$ given in Proposition \ref{cf1}, is also a continuous family of augmentations of $\mu.$ In fact, $\mathcal{C}$ is equivalent to $\mathcal{A}$. By Remark \ref{cfr}, without loss of generality, we can assume that $\deg(\rho_i)=\deg(\mathcal{C})$ for every $i\in\mathbf{A},$ which implies that $\deg(\mathcal{C})=\deg(\Phi(\mu,\nu))=\deg(\mathcal{A})$. Since the family $\mathcal{A}$ is independent of the choice of any key polynomial (see Remark \ref{r1}), $\mathcal{C}$ is a subset of $\mathcal{A}.$ Now, by Lemma 3.6 of \cite{EN1}, we have that $\mathcal{C}$ is a cofinal subset of $\mathcal{A}$. Also, by the uniqueness of limit augmentation (see \cite[Lemma 3.7]{EN1}),  \[\nu=[\mathcal{C};\phi,\gamma]=[\mathcal{A};\phi,\gamma].\]
	\end{remark}

\vspace{.20pt}

In 2021, Nart \cite{EN1} introduced the Mac Lane-Vaqui\'e chain for an arbitrary valuation $\mu$ on $K[x].$

\begin{definition}
	A finite or countably infinite chain of mixed augmentations 
	\begin{align*}
			(v\xrightarrow{\phi_0,\gamma_0})\mu_0\xrightarrow{\phi_1,\gamma_1} \mu_1\xrightarrow{\phi_2,\gamma_2}\cdots \longrightarrow\mu_{n}\xrightarrow{\phi_{n+1},\gamma_{n+1}} \mu_{n+1}\longrightarrow\cdots
	\end{align*} is called a {\bf Mac Lane-Vaqui\'e  chain}, if  $\mu_0$ is a depth zero valuation defined by $(\alpha_0,$ $\gamma_0),$ where $\phi_0=x-\alpha_0,$ and every augmentation step $\mu_{n}\rightarrow \mu_{n+1}$ is one of the following two types:
	\begin{enumerate}[(i)]
		\item Ordinary augmentation: $\mu_{n+1}=[\mu_n; \phi_{n+1},\gamma_{n+1}]$, such that $\deg (\mu_{n})<\deg(\Phi(\mu_{n},\mu_{n+1}))(=\deg(\mu_{n+1})).$
		\item  Limit augmentation: $\mu_{n+1}=[\mathcal{C};\phi_{n+1},\gamma_{n+1}]$, such that $\mathcal{C}$ is an essential continuous family of augmentations of $\mu_n$,  $\deg (\mu_{n})=\deg(\Phi(\mu_{n},\mu_{n+1}))<\deg(\mu_{n+1})$ and $\phi_{n}\notin\Phi(\mu_{n},\mu_{n+1}).$
	\end{enumerate}
	
\end{definition}
The following result categorizes MLV chains of extensions of $v$ to $K[x].$
\begin{theorem}(\cite[Theorem 4.3]{EN1})\label{mlvt1}
	Every valuation $\mu$ on $K[x]$ falls in one of the following cases. 
	\begin{enumerate}[(i)]
		\item It is the last valuation of a finite MLV chain
		\begin{align*}
			\mu_0\xrightarrow{\phi_1,\gamma_1} \mu_1\xrightarrow{\phi_2,\gamma_2}\cdots\longrightarrow \mu_{n}\xrightarrow{\phi_{n+1},\gamma_{n+1}} \mu_{n+1}\cdots\xrightarrow{\phi_r,\gamma_r} \mu_r=\mu.
		\end{align*}
		\item There exists a valuation $\mu_r$ falling in case (i) such that $\mu$  is the stable limit of a continuous family ${\mathcal{C}}$ of augmentations of $\mu_r$
		\begin{align*}
			\mu_0\xrightarrow{\phi_1,\gamma_1} \mu_1\xrightarrow{\phi_2,\gamma_2}\cdots\longrightarrow\mu_{n}\xrightarrow{\phi_{n+1},\gamma_{n+1}} \mu_{n+1}\cdots\xrightarrow{\phi_r,\gamma_r} \mu_r\xrightarrow {{\mathcal{C}}} \rho_{\mathcal{C}}=\mu,
		\end{align*}
		such that $\deg(\Phi(\mu_r,\mu))=\deg(\mu_r)$ and $\phi_r\notin \Phi(\mu_r, \mu).$
		\item It is the stable limit of an infinite MLV chain 
		\begin{align*}
			\mu_0\xrightarrow{\phi_1,\gamma_1} \mu_1\xrightarrow{\phi_2,\gamma_2}\cdots \longrightarrow \mu_{n}\xrightarrow{\phi_{n+1},\gamma_{n+1}} \mu_{n+1}\longrightarrow\cdots .
		\end{align*}
	\end{enumerate}
	We say that the MLV chain of  $\mu$ has finite depth $r,$ quasi-finite depth $r,$ or infinite depth, respectively.
\end{theorem}

We call a valuation $\mu$ of {\bf Type 1}, {\bf Type 2}, or {\bf Type 3} respectively, according as $\mu$ falls under cases (i), (ii), or (iii) as above. 
	Note that any valuation $\mu$ is of Type 1  if and only if $\mu$ is either valuation-transcendental or has a nontrivial support. Otherwise, $\mu$ is valuation-algebraic (see \cite[Lemma 4.5]{EN1}).

 For any MLV chain of $\mu$, the {\bf relative gap} of an augmentation step $\mu_{n}\longrightarrow \mu_{n+1}$ is defined as the rational number
\[d_{n}:=\frac{\deg(\mu_{n+1})}{\deg(\Phi(\mu_{n},\mu_{n+1}))}= \begin{cases} 1 & \text{if} \ \mu_{n}\longrightarrow \mu_{n+1}\ \text{is ordinary }\\ \deg(\mu_{n+1})/\deg(\mu_{n})\ &  \text{if} \ \mu_{n}\longrightarrow \mu_{n+1} \ \text{is limit }\end{cases} .\] Clearly, $d_n\geq 1$ for every $n.$ In fact, $d_n>1$ if $\mu_{n}\longrightarrow \mu_{n+1}$ is limit augmentation. Moreover, the degrees $\deg (\mu_n),$ depth, relative gaps and the nature of the augmentation steps $\mu_{n}\longrightarrow \mu_{n+1}$ are independent of the choice of an MLV chain of $\mu$ (see \cite[Corollary 4.4, Theorem 4.7]{EN1}).

\subsection{Extensions of valuations}

Let ($\overline{K},\bar{v})$ be as before. If $K^{sep}$ is the separable closure of $K$ in $\overline{K}$, then the fixed field of the {\bf decomposition group}\[D_{\bar{v}}=\{\sigma\in \text{Gal}(K^{sep}|K)\mid\bar{v}\circ\sigma=\bar{v}\},\] denoted by $K^h,$ is called the {\bf Henselization} of $K.$ Let $v^h$ be the restriction of $\bar{v}$ to $K^h.$ Since $\overline{K}|K^{sep}$ is purely inseparable, for all $\sigma\in \text{Aut}(\overline{K}|K),$ we have \[\bar{v}\circ\sigma=\bar{v}\iff \sigma\in \text{Aut}(\overline{K}|K^h).\] Hence, the valuation $v^h$ has a unique extension to $\overline{K}.$

A valued field $(K,v)$ is called {\bf Henselian} if $v$ extends uniquely to every finite extension $L$ of $K,$ i.e., $K=K^h.$ Any finite extension $L$ of $K$ is called {\bf unibranched} if $v$ has a unique extension to $L.$

	\begin{remark}\label{r21}
	For any irreducible polynomial $g$ over $K,$ the extensions of $v$ to the simple field extension $K[x]/(g)$ are in one-to-one correspondence with the irreducible factors of $g$ in $K^h[x]$ (see \cite[Section 17]{O} and \cite[Section 3]{EN3}). More precisely, if $g=G_1G_2\cdots G_s$ is the factorization of $g$ into irreducibles over $K^h,$ then there are exactly $s$ extensions of $v$ to $K[x]/(g)$, given by ${v}_i=\bar{v}\circ \lambda_{\theta_i},$ where $ \lambda_{\theta_i}(x+gK[x])=\theta_i$, and $\theta_i\in Z(G_i)$\footnote{For any $f\in \overline{K}[x],$ $Z(f)$ is the set of all roots of $f$ in $\overline{K}.$}, for all $i,\ 1\leq i \leq s.$
	Moreover, there are  exactly $'s'$ valuations on $K[x]$ with nontrivial support $gK[x]$ induced by ${v}_i$, defined as 
	\[v_{G_i}:K[x]\twoheadrightarrow K[x]/(g)\xrightarrow{{v}_i} \Gamma_{\bar{v}}\infty.\]
\end{remark}

If $\mu$ is any extension of $v$ to $K[x],$ then there exists a unique common extension $\mu^h$ of $\mu$ and $v^h$ to $K^h[x]$ (see \cite[Theorem A]{EN5}). In particular, we have 

\begin{proposition}(\cite[Proposition 5.6]{EN3})\label{df1}
Let $\mu$ be a valuation-transcendental extension. Then for any $\phi\in\operatorname{KP}(\mu),$ there exist a unique irreducible factor $\phi^h$ of $\phi$ in $K^h[x]$ such that $\phi^h\in \operatorname{KP}(\mu^h)$.
Moreover, we have $\deg_\mu(\phi)=\deg_{\mu^h}(\phi^h).$ In particular, if $\phi$ has minimal degree in $\operatorname{KP}(\mu),$ then $\phi^h$ has minimal degree in $\operatorname{KP}(\mu^h).$
\end{proposition}

The first part of the following lemma is proved as Proposition 3.7 in \cite{NA}, and the second part follows directly from Remark \ref{r21}.
\begin{lemma}\label{l1}
	Let $\mu$ be a valuation-transcendental extension. For any $\phi , \psi\in \operatorname{KP}(\mu)$ and $\alpha\in Z(\phi)$, we have
	\begin{enumerate}[(i)]
		\item $\phi=\phi^h \iff \psi=\psi^h,$
		\item $\phi=\phi^h \iff K(\alpha)|K $ is unibranched.
	\end{enumerate} 
\end{lemma}

\begin{lemma}(\cite[Lemma 5.7]{EN3})\label{df2}
Let $\mu$ be a valuation-transcendental extension and $\phi\in \operatorname{KP}(\mu).$ If $\nu=[\mu;\phi,\gamma]$ is an ordinary augmentation, for some $\gamma\in\Gamma\infty$ such that $\gamma>\mu(\phi),$ then \[\nu^h=[\mu^h;
\phi^h,\gamma^h], \ \text{where}\ \gamma^h=\gamma-\mu^h(\phi/\phi^h).\]
\end{lemma}

Let $\mathcal{C}=(\rho_i)_{i\in\mathbf{A}}$ be a continuous family of valuations on $K[x]$, and let $\mathcal{C}^h=(\rho_i^h)_{i\in\mathbf{A}}$ be the corresponding family of canonical extensions to $K^h[x]$. By \cite[Theorem A]{EN5}, $\mathcal{C}^h$ is a totally ordered family of valuations on $K^h[x]$ without a maximal element, parametrized by $\mathbf{A}$. Since the degree of valuations admitting key polynomials is an order-preserving function (see \cite[Lemma 2.2]{Ma2}), we have $\deg(\rho_i^h)\leq \deg(\rho_j^h)$ for all $i<j\in\mathbf{A}$.  By Proposition \ref{df1}, $\deg(\rho_i^h)\leq\deg(\rho_i)\leq \deg(\mathcal{C})$, for all $i\in \mathbf{A}$. Therefore, $\mathcal{C}^h$ has a stable degree and is a continuous family of valuations on $K^h[x]$.

\begin{lemma}(\cite[Lemma 5.13]{EN3})\label{df3}
Let $\mathcal{C}  =(\rho_i)_{i\in\mathbf{A}}$ be a continuous family of valuations on $K[x]$ and let $\mathcal{C}^h$ be its canonical extensions to $K^h[x].$ Let $\phi$ be a limit key polynomial for $\mathcal{C}$ and consider the limit augmentation $\nu=[\mathcal{C};\phi,\gamma],$ for some $\gamma\in\Gamma\infty$ such that $\gamma>\rho_i(\phi)$ for all $i\in\mathbf{A}.$ Then, $\phi^h$ is a limit key polynomial for $\mathcal{C}^h$, $\phi/\phi^h$ is $\mathcal{C}^h$-stable and 
\[\nu^h=[\mathcal{C}^h;\phi^h,\gamma^h], \ \text{where}\ \gamma^h=\gamma-(\rho_{\mathcal{C}})^h(\phi/\phi^h).\]
\end{lemma}

In the next lemma, we consider the case when the continuous family $\mathcal{C}$ does not admit a limit key polynomial.
\begin{lemma}\label{ml1}
Let $\mathcal{C}  =(\rho_i)_{i\in\mathbf{A}}$ be a continuous family of valuations on $K[x]$ and let $\mathcal{C}^h$ be its canonical extensions to $K^h[x].$ If $\mathcal{C}$ has a stable limit $\rho_{\mathcal{C}}$, then the family $\mathcal{C}^h$ also has a stable limit $\rho_{\mathcal{C}^h}$ and  $(\rho_{\mathcal{C}})^h=\rho_{\mathcal{C}^h}.$
\end{lemma}
\begin{proof}
Assume, to the contrary, that $\mathcal{C}^h$ does not have a stable limit, and let $\phi\in K^h[x]$ be $\mathcal{C}^h$-unstable. As the product of $\mathcal{C}^h$-stable polynomials is also $\mathcal{C}^h$-stable, we can assume that $\phi$ is monic and irreducible over $K^h.$ Also $\phi\nin K[x]$ because all polynomials in $K[x]$ are $\mathcal{C}$-stable. 

Let $\alpha\in Z(\phi)$ and let $f$ be the minimal polynomial of $\alpha$ over $K.$ Then $\phi\mid f$ in $K^h[x],$ say, $f=\phi g$ for some $g\in K^h[x]$.  For every $i<j\in\mathbf{A},$ we have $\rho_i^h(\phi)<\rho_j^h(\phi),$ which implies that $\rho_i^h(f)<\rho_j^h(f)$. But $f\in K[x],$ so \[\rho_i(f)=\rho_i^h(f)<\rho_j^h(f)=\rho_j(f)\] for every $i<j\in\mathbf{A},$ contradicting the fact that $\mathcal{C}$ has a stable limit. Hence, $\mathcal{C}^h$ has a stable limit $\rho_{\mathcal{C}^h}$.

Since each $f\in K[x]$ is $\mathcal{C}^h$-stable, there exists $i_0\in\mathbf{A}$ such that \[\rho_i^h(f)=\rho_{i_0}^h(f)=\rho_{\mathcal{C}^h}(f)\] for every $i\geq i_0$ in $\mathbf{A}.$ But $f\in K[x]$ implies that $\rho_i(f)=\rho_{i_0}(f)=\rho_{\mathcal{C}^h}(f)$ for every $i\geq i_0$ in $\mathbf{A}.$ Since $f$ is also $\mathcal{C}$-stable,  $\rho_{\mathcal{C}^h}(f)=\rho_{\mathcal{C}}(f),$ i.e., $\rho_{\mathcal{C}^h}\mid_{K[x]}=\rho_{\mathcal{C}}.$ As $\rho_\mathcal{C}$ has a unique extension to $K^h[x]$, we have $(\rho_{\mathcal{C}})^h=\rho_{\mathcal{C}^h}.$ 
\end{proof}

\begin{remark}\label{rem2}
	Let $\mathcal{C}=(\rho_i)_{i\in\mathbf{A}}$ be a continuous family of augmentations of $\mu$ on $K[x]$, with $\rho_i=[\mu; \chi_i,\beta_i]$. By Lemma \ref{df2}, \[\rho_i^h=[\mu^h;\chi_i^h,\beta_i^h],\ \text{where} \  \beta_i^h=\beta_i-\mu^h(\chi_i/\chi_i^h).\] Then $\mathcal{C}^h=(\rho_i^h)_{i\in\mathbf{A}}$ is also a continuous family of augmentations of $\mu^h$ on $K^h[x]$. 
	
	Moreover, if $\deg(\mu)=\deg(\mathcal{C})$ then $\deg(\mu^h)=\deg(\mathcal{C}^h).$ Indeed, for any minimal degree key polynomial $\phi$ for $\mu$, since each $\chi_i\in\operatorname{KP}(\mu),$ arguing as in the proof of Proposition 3.7 of \cite{NA}, we have  \[\frac{\deg\phi}{\deg\phi^h}=\frac{\deg\chi_i}{\deg\chi_i^h}.\] Now $\deg(\mu)=\deg(\mathcal{C})$ implies  $\deg\phi=\deg\chi_i$, and then from the above equation, we have $\deg\phi^h=\deg\chi_i^h,$ for every $i\in\mathbf{A}.$  By Proposition \ref{df1}, $\deg(\mu^h)=\deg\phi^h$, hence,  $\deg(\mu^h)=\deg(\mathcal{C}^h).$
\end{remark}

	Let $\overline{\mu}$ be any common extension of $\mu$ and $\bar{v}$ to $\overline{K}[x].$ For any polynomial $f$ in $K[x],$ we define the {\bf optimizing value} of $f$ as
$$\delta(f):=\max\{\overline{\mu}(x-\alpha)\mid \alpha\in Z(f)\}.$$ A root $\alpha$ of $f$ is said to be an {\bf optimizing root} of $f$ if $\overline{\mu}(x-\alpha)=\delta(f).$
This optimizing value of $f$ is independent of the choice of the common extension $\overline{\mu}.$ In particular, we have

  \begin{theorem}(\cite[Corollary 3.13]{EJP})\label{t21}
 	Let $\mu$ be a valuation-transcendental extension of $v$ to $K[x],$ and $\phi\in \operatorname{KP}(\mu).$ Then the set of optimizing roots of $\phi$ is equal to $Z(\phi^h).$
 \end{theorem}
 
 In 2026, Nart  gave an explicit description of common extensions in terms of key polynomials of $\mu.$
 \begin{theorem}(\cite[Remark 3.14]{EJP})\label{t2}
 	Let $\mu$ be valuation-transcendental, and $\phi\in\operatorname{KP}(\mu).$ Then for any common extension $\overline{\mu}$ of $\mu$ and $\bar{v}$ to $\overline{K}[x]$, we have $\overline{\mu}=\overline{w}_{\alpha,\delta},$ where $\alpha\in Z(\phi^h)$ and $\delta=\delta(\phi).$ Moreover, $\delta=\delta(\phi)$ is independent of the choice of the key polynomial.
\end{theorem}

 	\subsection{Defect}
 
 Let $(K,v)$ be a valued field, $L|K$ a finite extension and $v_L$ an extension of $v$ to $L.$ The {\bf Henselian-defect} of the extension $v_L|v$ is defined as \[d(v_L|v)=\frac{[L^h:K^h]}{ef},\] where $e$ and $f$ denote respectively the ramification index and inertia degree of $v_L|v,$  and $L^h, K^h$ are the Henselizations of $L$ and $K$  with respect to the extension of $v_L$ to $\overline{K}.$  If $L|K$ is unibranched, then $v_L$ is the unique extension of $v$ to $L.$ In this case, we denote $d(v_L|v)$ as $d(L|K).$ 
 
 A finite extension $L|K$ is said to be {\bf defectless} if $d(v_L|v)=1$ for every extension $v_L$ of $v$ to $L.$ A valued field $(K,v)$ is said to be defectless if every finite extension of $K$ is defectless.

 \begin{lemma}(\cite{V0} and \cite[Lemma 6.1]{EN3})\label{def}
 	Let $\mu\longrightarrow \nu$ be an augmentation and $\Phi(\mu,\nu)$ be the corresponding tangent direction, equipped with the pre-ordering determined by the action of $\nu.$ For $Q\in\Phi(\mu,\nu) ,$ set $\rho_Q=[\mu;Q, \nu(Q)].$ Let $\phi \in K[x]$ be either a key polynomial of minimal degree for $\nu,$ or $supp(\nu)=\phi K[x].$ Then the positive integer \[d=\min\{\deg_{\rho_Q}(\phi) \mid Q\in \Phi(\mu,\nu)\}\] is independent of the choice of $\phi.$ Moreover, the set of all $Q\in \Phi(\mu,\nu)$ such that $\deg_{\rho_Q}(\phi)=d$ is cofinal in $\Phi(\mu,\nu).$
 \end{lemma}
 This stable value $d$ is called the {\bf defect} of the augmentation and is denoted by $d(\mu\longrightarrow \nu).$ 
 
 \begin{lemma}(\cite[Lemma 6.3]{EN3})\label{nl1}
 	If $\mu\longrightarrow \nu$ is an ordinary augmentation, then $d(\mu\longrightarrow \nu)=1.$
 \end{lemma}
 \begin{lemma}(\cite[Lemma 6.4]{EN3})\label{nl2}
 	Suppose $(K,v)$ is Henselian. For all augmentations $\mu\longrightarrow \nu$ we have
 	\[d(\mu\longrightarrow \nu)=\deg(\nu)/ \deg (\Phi(\mu,\nu)).\]
 \end{lemma}

 \begin{theorem}(\cite[Theorem 1.1]{EN3})
 	Let $\mu\longrightarrow\nu$ be an augmentation of valuations on $K[x]$ and $\mu^h\longrightarrow\nu^h$ the corresponding augmentation on $K^h[x].$ Then\[d(\mu\longrightarrow\nu)=d(\mu^h\longrightarrow\nu^h).\]
 \end{theorem}
 From the above two results, it follows that 
 
 \begin{proposition}\label{dh}
 	For any augmentation $\mu\longrightarrow\nu$, we have 
 	\[d(\mu\longrightarrow\nu)=\deg(\nu^h)/ \deg (\Phi(\mu^h,\nu^h)).\]
 \end{proposition}
 
 The following result shows that the defect of any limit augmentation can also be defined by using any limit key polynomial.
 \begin{proposition}
 	For any limit augmentation $\mu\longrightarrow \nu$ such that $\nu=[\mathcal{C};\phi,\gamma]$ for an essential continuous family $\mathcal{C}=(\rho_i)_{i\in\mathbf{A}}$ of augmentations of $\mu$, we have \[d(\mu\longrightarrow\nu)=d(\mu\longrightarrow\nu_{\mathcal{C}}),\] where $\nu_{\mathcal{C}}$ is the minimal limit augmentation of  $\mathcal{C}$. Moreover, it is independent of the choice of any limit key polynomial of $\mathcal{C}.$
 \end{proposition}
 \begin{proof}
 	As $\nu=[\mathcal{C};\phi,\gamma]$ is a limit augmentation, $\phi$ is a minimal degree key polynomial for $\nu.$ Therefore, by the definition of defect, we have
 	\begin{align}\label{e1}
 		d(\mu\longrightarrow\nu)=\min\{\deg_{\rho_Q}(\phi) \mid Q\in \Phi(\mu,\nu)\},
 		\end{align}
 		where $\rho_Q=[\mu;Q,\nu(Q)]$ for every $Q\in\Phi(\mu,\nu).$ Also, we know that $\nu_{\mathcal{C}}=[\mathcal{C};\phi,\gamma_{\mathcal{C}}]$ is the minimal limit augmentation for $\mathcal{C},$ where $\gamma_{\mathcal{C}}=\sup\{\rho_i(\phi)\mid i\in\mathbf{A}\}.$ If $\gamma=\gamma_\mathcal{C}$ then $\nu=\nu_\mathcal{C}$ and hence, $d(\mu\longrightarrow\nu)=d(\mu\longrightarrow\nu_{\mathcal{C}}).$
 		
 		 Otherwise, for $\gamma>\gamma_\mathcal{C},$ by Lemma \ref{vt1}, we have that $\nu=[\nu_{\mathcal{C}};\phi,\gamma]$ is an ordinary augmentation. Here, $\mu<\nu_\mathcal{C}<\nu;$ therefore, by Corollary \ref{2.5}, \[\Phi(\mu,\nu)=\Phi(\mu,\nu_\mathcal{C}).\] As $\deg(\Phi(\mu,\nu))=\deg(\mathcal{C})<\deg\phi$, for any $Q\in\Phi(\mu,\nu),$ $\deg Q<\deg\phi$ implies that \[\nu(Q)=\nu_{\mathcal{C}}(Q) \ \text{and}\  \rho_Q=[\mu;Q,\nu_\mathcal{C}(Q)].\] Then, for the minimal limit augmentation $\mu\longrightarrow\nu_\mathcal{C},$ we have \begin{align*}
 			d(\mu\longrightarrow\nu_\mathcal{C})=\min\{\deg_{\rho_Q}(\phi) \mid Q\in \Phi(\mu,\nu)\}.
 		\end{align*}
 		On comparing the above equation with (\ref{e1}), we obtain
 		\[d(\mu\longrightarrow\nu)=d(\mu\longrightarrow\nu_{\mathcal{C}}).\] 
 		
 		Since $d(\mu\longrightarrow\nu_{\mathcal{C}})$ is independent of the choice of a minimal degree key polynomial for $\nu_\mathcal{C},$ and $\operatorname{KP}(\nu_\mathcal{C})=\operatorname{KP}_{\infty}(\mathcal{C}),$ we have that, $d(\mu\longrightarrow\nu)$ is independent of the choice of any limit key polynomial $\phi\in\operatorname{KP}_{\infty}(\mathcal{C}).$
  \end{proof}
  
  Using the above proposition, we now prove that the defect of the augmentation steps is independent of the choice of an MLV chain of $\mu.$
  \begin{theorem}\label{td1}
  	The defect of an augmentation step in an MLV chain of valuation $\mu$ is an intrinsic property.
  \end{theorem}
  \begin{proof}
  	Consider two different MLV chains for $\mu$: \[(v\xrightarrow{\phi_0})\mu_0\xrightarrow{\phi_1} \mu_1\xrightarrow{\phi_2}\cdots \longrightarrow \mu_{n}\xrightarrow{\phi_{n+1}} \mu_{n+1}\longrightarrow\cdots\longrightarrow \mu,\] and
  	\[(v\xrightarrow{\psi_0})w_0\xrightarrow{\psi_1} w_1\xrightarrow{\psi_2}\cdots \longrightarrow w_{n}\xrightarrow{\psi_{n+1}} w_{n+1}\longrightarrow\cdots\longrightarrow \mu.\]    To prove the result, it is enough to show that for every $n,$ \[d(\mu_n\longrightarrow \mu_{n+1})=d(w_n\longrightarrow w_{n+1}).\] By the uniqueness of MLV chains (see \cite[Theorem 4.7]{EN1}), we have that $\mu_n\longrightarrow \mu_{n+1}$ is ordinary (or limit) if and only if $w_n\longrightarrow w_{n+1}$ is ordinary (or limit). 
  	
  	If $\mu_n\longrightarrow \mu_{n+1}$ is ordinary, then by Lemma \ref{nl1}, \[d(\mu_n\longrightarrow \mu_{n+1})=1=d(w_n\longrightarrow w_{n+1}).\]
  	
  	Assume now that  $\mu_n\longrightarrow \mu_{n+1}$ is a limit augmentation. Let $\mathcal{C}$ and $\mathcal{C}'$ be the defining essential continuous families for the limit augmentations $\mu_n\longrightarrow \mu_{n+1}$ and  $w_n\longrightarrow w_{n+1},$ respectively. By the uniqueness of MLV chains, $\mathcal{C}$ and $\mathcal{C}'$ are equivalent, i.e., both families have the same stable and unstable polynomials (see \cite[Lemma 3.6]{EN1}). Therefore, $\phi_{n+1},\psi_{n+1}\in\operatorname{KP}_{\infty}(\mathcal{C})=\operatorname{KP}_{\infty}(\mathcal{C}').$ Using Remark \ref{lmr}, we have that  $\mathcal{A}=(\rho_{\alpha})_{\alpha\in\mathbb{A}}$ and $\mathcal{A}'=(\rho_{i})_{i\in\mathbb{A}'}$ are essential continuous families of augmentations of $\mu_n$ and $w_n,$ respectively, where, \[\mathbb{A}=\mu(\Phi(\mu_n,\mu_{n+1})),\   \rho_\alpha=[\mu_n;\chi_\alpha,\alpha]\  \text{with}\  \chi_\alpha\in\Phi(\mu_n,\mu_{n+1}),\ \mu(\chi_\alpha)=\alpha;\] \[ \mathbb{A}'=\mu(\Phi(w_n,w_{n+1})),\ \rho_i=[w_n;\chi_i,i]\ \text{with}\ \chi_i\in\Phi(w_n,w_{n+1}),\ \ \mu(\chi_i)=i.\]  As $\mathcal{A}$ and $\mathcal{A}'$ are independent of the choice of key polynomials, by the definition of defect together with the above proposition, we have
  	\[d(\mu_n\longrightarrow \mu_{n+1})=\min\{\deg_{\rho_\alpha}(\psi_{n+1}) \mid \rho_\alpha\in \mathcal{A}\},\ \text{and}\] 
  	\[d(w_n\longrightarrow w_{n+1})=\min\{\deg_{\rho_i}(\phi_{n+1}) \mid \rho_i\in \mathcal{A}'\}.\] Put $d=d(\mu_n\longrightarrow \mu_{n+1})$ and $d'=d(w_n\longrightarrow w_{n+1})$. By Lemma \ref{def}, the set of all $\rho_\alpha$ with $\deg_{\rho_\alpha}(\psi_{n+1})=d$ forms a cofinal subset $\mathcal{A}_d$ of $\mathcal{A}$.  Since a cofinal subset of a continuous family is again a continuous family for the same valuation, $\mathcal{A}_d$ is a continuous family of augmentations of $\mu_n.$ Similarly, we obtain a continuous family $\mathcal{A}'_{d'}$ of augmentations of $w_n.$ Again, by Lemma 3.6 of \cite{EN1}, $\mathcal{A}_d$ and $\mathcal{A}'_{d'}$ are equivalent, i.e., they are cofinal in each other. 
  	
  	Since $w_n<\rho_i$ for some $\rho_i\in\mathcal{A}'_{d'}$ and $\mathcal{A}_d$ is cofinal in $\mathcal{A}'_{d'},$ there exists $\rho_\alpha\in\mathcal{A}_d$ such that $\rho_i<\rho_\alpha$. Then for any $\beta\in \mathbb{A}$ with $\beta>\alpha$ and $\rho_\beta\in\mathcal{A}_d$, we have \[w_n(\chi_\beta)\leq \rho_i(\chi_\beta)\leq\rho_\alpha(\chi_\beta)=\alpha<\beta=\rho_\beta(\chi_\beta)=\mu(\chi_\beta)=w_{n+1}(\chi_\beta),\] i.e., $w_n(\chi_\beta)<w_{n+1}(\chi_\beta).$ As \[\deg\chi_\beta=\deg(\mathcal{A}_d)=\deg(\mathcal{A}'_{d'})=\deg(\Phi(w_n,w_{n+1})),\] $\chi_\beta\in\Phi(w_n,w_{n+1})$ and consequently, $\chi_\beta$ is a key polynomial for $w_n.$ For any polynomial $f\in K[x]$ with $\deg f<\deg(\mu_n)=\deg(w_n),$ we have $\mu_n(f)=\mu(f)=w_n(f),$ which implies that  \[\rho_\beta=[\mu_n;\chi_\beta,\beta]=[w_n;\chi_\beta,\beta].\] Hence, $\rho_\beta\in\mathcal{A}'$ and \[\deg_{\rho_\beta}(\psi_{n+1})\geq d'=\min\{\deg_{\rho_i}(\psi_{n+1})\mid \rho_i\in\mathcal{A}'\}.\] Whereas, by the choice of $\beta$, $\deg_{\rho_\beta}(\psi_{n+1})=d$ and therefore,  $d\geq d'.$ A symmetric argument shows that $d'\geq d,$ thus $d=d'$.
   \end{proof}
 
 	Let $\mu$ be either valuation-transcendental or have a nontrivial support, i.e., $\mu$ is of Type 1 with an MLV chain (say)
 	\begin{align}\label{m1}
 				(v \longrightarrow) \mu_0  \longrightarrow \mu_1  \longrightarrow \cdots  \longrightarrow \mu_{r-1}  \longrightarrow \mu_r=\mu ,
 				\end{align}
 			then the defect of $\mu|v,$ given in \cite{EN3}, is defined as
 			\[d(\mu|v):=d(\mu_0\longrightarrow \mu_1) \cdots d(\mu_{r-1}\longrightarrow \mu).\]
 			Let $\phi$ be either a minimal degree key polynomial for $\mu$ or $\text{supp}(\mu)=\phi K[x]$, and let $L=K(\theta)$ for any $\theta\in Z(\phi).$  
 			
 			In 2007, Vaqui\'e gave a characterization for the defect of a unibranched simple algebraic extension $L$.
 			\begin{theorem}(\cite[Corollary 2.10]{V2})\label{vth1}
 				Let $v_L$ be the unique extension of $v$ to a simple algebraic extension $L,$ and $\mu$ be the valuation on $K[x]$ induced by $v_L.$ Then the defect of the extension $v_L|v$ is the product of the relative gaps of any MLV chain of $\mu.$
 			\end{theorem} 	Later in 2023, Nart and Novacoski gave another characterization for the defect of a simple algebraic extension, which generalizes the above for the Henselian case (see Lemma \ref{nl2}).
 			
 			\begin{theorem}(\cite[Theorem 6.14]{EN3})\label{T2}
 			Let $v_L$ be an extension of $v$ to a finite simple extension $L|K.$ Let $\mu$ be the valuation on $K[x]$ induced by $v_L.$ For any MLV chain for $\mu=\mu_r$ (\ref{m1}), we have\[d(v_L|v)=d(\mu|v)=d(\mu_0\longrightarrow \mu_1)\cdots d(\mu_{r-1}\longrightarrow \mu_r).\]
 			\end{theorem}
 		 Recently, in 2026, this result was extended to valuation-transcendental extensions (see \cite[Theorem 4.10]{NA}). Therefore, in either case, we have $d(\mu|v)=d(v_L|v),$ the Henselian defect. Hence, using the above results, it follows that if $K(\theta)|K$ is unibranched, then  
 			\begin{align}\label{ev}
 				d(K(\theta)|K)=\prod_{n=0}^{r-1} d(\mu_n\longrightarrow \mu_{n+1})=\prod_{n=0}^{r-1}\frac{\deg(\mu_{n+1})}{\deg(\Phi(\mu_n,\mu_{n+1})}.
 				\end{align}

 	\subsection{Abstract Key Polynomials}
 		Abstract key polynomials were first introduced by Decaup et al.\  in \cite{DSM}, as an alternative definition of key polynomials.
 			 The relation between key polynomials and abstract key polynomials is explored in \cite{Ma, DSM,NS}. 
 			
 			\begin{definition}
 				A monic polynomial $Q$ in $K[x]$ is said to be an {\bf abstract key polynomial}  (abbreviated  ABKP) for $\mu$ if for every polynomial $f$ in $K[x]$  with $\deg f< \deg Q,$ we have $\delta(f)<\delta(Q).$
 			\end{definition}
 			It is immediate from the definition that all monic linear polynomials are ABKPs for $\mu.$  Also an ABKP for $\mu$ is irreducible. If $\mu$ has nontrivial support $\phi K[x],$ then $\phi$ is a highest degree ABKP for $\mu$ since $\delta(\phi)=\infty.$
 			Note that $Q$ is an ABKP for $\mu$ if and only if $(\alpha,\delta(Q))$ is a minimal pair for $\mu,$ for some optimizing root $\alpha$ of $Q$ (see \cite[Theorem 1.1]{JN}).
 			\begin{definition}
 				Let $\mu$ be a valuation on $K[x].$ Then
 				for a polynomial $Q$ in $K[x]$ the {\bf $Q$-truncation} of $\mu$ is a map $\mu_Q:K[x]\longrightarrow \Gamma_\mu$ defined by 
 				$$ \mu_Q(f):= \min_{i\geq 0}\{\mu(f_iQ^i)\},$$
 				where $\sum_{i\geq 0} f_i Q^i,$ $\deg f_i <\deg Q,$  is the $Q $-expansion of $f.$ 
 			\end{definition}
 			The $Q$-truncation  $\mu_Q$  of $\mu$ need not be a valuation \cite[Example 2.5]{NS}. However,  if $Q$ is an ABKP for $\mu,$ then $\mu_Q$ is a valuation on $K[x]$ (see \cite[Proposition 2.6]{NS}). If $Q\in \text{supp}(\mu),$ then $\mu_Q=\mu,$ otherwise, $\mu_Q$ is a valuation-transcendental extension and $Q$ is a minimal degree key polynomial for $\mu_Q$ (see \cite[Theorem 2.21]{Ma}). In either case, $Q$ is also an ABKP for $\mu_Q$ (\cite[Corollary 2.22]{Ma}).
 			\begin{definition}
 				A family $\Lambda=\{Q_i\}_{i\in\Delta}$ of ABKPs for $\mu,$ indexed by a well-ordered set $\Delta,$ is said to be a {\bf complete sequence of ABKPs} for $\mu$ if the following conditions hold:
 				\begin{enumerate}[(i)]
 					\item $\delta(Q_i)\neq \delta(Q_j)$ for every $i\neq j\in\Delta.$
 					\item $\Lambda$ is well-ordered with respect to the ordering given by $Q_i< Q_j$ if and only if  $\delta(Q_i)<\delta(Q_j)$ for every $i<j\in \Delta.$ 
 					\item For any $f\in K[x],$ there exists some $Q_i \in \Lambda$ such that $\deg Q_i\leq \deg f$ and  $\mu_{Q_i}(f)=\mu(f).$
 				\end{enumerate}
 			\end{definition}
 			
 			
 			\begin{remark}\label{rabkp}
 				Every valuation $\mu$ on $K[x]$ admits a complete sequence of ABKPs for $\mu$ (see \cite[Theorem 1.1]{NS}). As shown in \cite[Remark 4.6]{W}, there is a complete sequence  $\Lambda=\{Q_i\}_{i\in\Delta}$ of ABKPs  for  $\mu$ with the following properties:
 				\begin{enumerate}[(i)]
 					\item $\Delta=\bigcup_{j\in I}\Delta_j$ with $I=\{0,1,\ldots, N\}$ or $\mathbb{N}\cup\{0\},$ and for each $j\in I$ we have $\Delta_j=\{j\}\cup\vartheta_{j},$ where $\vartheta_j$ is an ordered set without a last element or is empty.
 					\item $Q_0$ is a monic polynomial of degree one.
 					\item For all $j\in I\setminus \{0\}$ and $i \in\vartheta_{j-1},$ we have $j-1<i<j.$ 
 					\item All  polynomials $Q_i$ with $i\in\Delta_j$ have the same degree and  have degree strictly  less than  the degree of the polynomials $Q_{i'}$ for every $i'\in\Delta_{j+1}.$
 					\item For each $i<i'\in\Delta$ we have $\mu(Q_i)<\mu(Q_{i'}).$
 				\end{enumerate}
 			\end{remark}
 			
 			\begin{definition}
 				Let $\Lambda=\{Q_i\}_{i\in\Delta}$ be a complete sequence of ABKPs for $\mu.$	For	an element $j\in I,$ we say that  $Q_{j+1}$ is a {\bf limit ABKP} for $\mu$ if $j+1$ has no immediate predecessor in $\Delta,$ i.e.,	$\vartheta_{j}\neq \emptyset.$\end{definition}

 \section{Regular MLV chains}
 
 In this section, we study in detail an MLV chain of a valuation $\mu,$ under the hypothesis that each defining key polynomial $\phi_n$ is irreducible over $K^h$, i.e., $\phi_n=\phi_n^h.$
 
 The first result of this section shows that under this hypothesis, the MLV chain of $\mu$ and its canonical extension $\mu^h$ to $K^h[x]$ have the same depth.
 \begin{theorem}
 	Let $\mu$ be any extension of $v$ to $K[x]$ with an MLV chain \begin{align*}
 		(v\xrightarrow{\phi_0,\gamma_0})\mu_0\xrightarrow{\phi_1,\gamma_1} \mu_1\xrightarrow{\phi_2,\gamma_2}\cdots \longrightarrow \mu_{n}\xrightarrow{\phi_{n+1},\gamma_{n+1}} \mu_{n+1}\longrightarrow\cdots\longrightarrow \mu. 
 	\end{align*}
 	If each $\phi_n$ is irreducible over $K^h,$ then \begin{align}\label{edf1}
 		(v^h\xrightarrow{\phi_0,\gamma_0})\mu_0^h\xrightarrow{\phi_1,\gamma_1} \mu_1^h\xrightarrow{\phi_2,\gamma_2}\cdots \longrightarrow \mu_{n}^h\xrightarrow{\phi_{n+1},\gamma_{n+1}} \mu_{n+1}^h\longrightarrow\cdots\longrightarrow \mu^h 
 	\end{align} is an MLV chain of $\mu^h.$
 	\end{theorem}
 	\begin{proof}
 	  As $\mu_0$ is a depth zero valuation, its canonical extension $\mu_0^h$ to $K^h[x]$ is also a depth zero valuation defined over $v^h$ by $\phi_0$ and $\gamma_0$. For each $n,$ since $\phi_n$ is a minimal degree key polynomial for $\mu_n,$ Proposition \ref{df1} implies that $\phi_n^h$ is also a minimal degree key polynomial for the canonical extension $\mu_n^h$ of $\mu_n$ and $v^h$ to $K^h[x].$ Therefore, $\deg(\mu_n^h)=\deg\phi_n^h.$
 	  
 	   For an arbitrary but fixed $n$, $\mu_n\longrightarrow \mu_{n+1}$ is an MLV step defined by $\phi_{n+1}$ and $\gamma_{n+1}.$ We now consider the following two cases:
 	   
 	  \smallskip
 		\noindent{\bf Case I:} The augmentation $\mu_n\longrightarrow \mu_{n+1}$ is ordinary. Then, by Lemma \ref{df2}, we have \[\mu_{n+1}^h=[\mu_n^h;\phi_{n+1}^h,\gamma_{n+1}^h],\ \text{where}\ \gamma_{n+1}^h=\gamma_{n+1}-\mu_n^h(\phi_{n+1}/\phi_{n+1}^h).\] By hypothesis, $\phi_{n+1}=\phi_{n+1}^h,$ which implies that $\gamma_{n+1}^h=\gamma_{n+1}$ and that $\mu_n^h\longrightarrow \mu_{n+1}^h$ is also an ordinary augmentation defined by $\phi_{n+1}$ and $\gamma_{n+1},$ with $\phi_{n+1}^h\in \Phi(\mu_n^h,\mu_{n+1}^h).$ 
 		
 		Since $\mu_n\longrightarrow \mu_{n+1}$ is an MLV step, \[\deg(\mu_n)<\deg(\Phi(\mu_n,\mu_{n+1}))=\deg(\phi_{n+1}).\] As $\deg(\mu_n)=\deg\phi_n$,  $\deg(\mu_n^h)=\deg\phi_n^h,$ $\phi_n=\phi_n^h$ and $\phi_{n+1}=\phi_{n+1}^h$, we have \[\deg(\mu_n^h)<\deg(\Phi(\mu_n^h,\mu_{n+1}^h)),\] which implies that $\mu_n^h\xrightarrow{\phi_{n+1},\gamma_{n+1}} \mu_{n+1}^h$ is also an MLV step.
 		
 		\smallskip
 		\noindent{\bf Case II:}  The augmentation $\mu_n\longrightarrow \mu_{n+1}$ is limit. Then, there exists an essential continuous family $\mathcal{C}  =(\rho_i)_{i\in\mathbf{A}}$ of augmentations of $\mu_n$ such that $\mu_{n+1}=[\mathcal{C}; \phi_{n+1},\gamma_{n+1}],$ where $\phi_{n+1}$ is a limit key polynomial for $\mathcal{C}.$ By Lemma \ref{df3}, we have \[\mu_{n+1}^h=[\mathcal{C}^h; \phi_{n+1}^h,\gamma_{n+1}^h],\ \text{where}\ \gamma_{n+1}^h=\gamma_{n+1}-(\rho_{\mathcal{C}})^h(\phi_{n+1}/\phi_{n+1}^h),\] $\mathcal{C}^h=(\rho_i^h)_{i\in \mathbf{A}}$ is a continuous family of augmentations of $\mu_n^h$ and $\phi_{n+1}^h$ is a limit key polynomial for $\mathcal{C}^h.$ Since $\phi_{n+1}=\phi_{n+1}^h,$ $\gamma_{n+1}^h=\gamma_{n+1},$ and $\mu_n^h\longrightarrow \mu_{n+1}^h$ is also a limit augmentation defined by $\phi_{n+1}$ and $\gamma_{n+1}.$
 		
 		 Since $\mu_n\longrightarrow \mu_{n+1}$ is an MLV step, \[\deg(\mu_n)=\deg(\Phi(\mu_n,\mu_{n+1}))<\deg(\mu_{n+1})\  \text{and}\  \phi_n\nin \Phi(\mu_n,\mu_{n+1}).\] As the family $\mathcal{C}$ is essential, $\deg(\mathcal{C})=\deg(\Phi(\mu_n,\mu_{n+1})).$ By Remark \ref{rem2}, $\deg(\mu_n)=\deg(\mathcal{C})$ implies that $\deg(\mu_n^h)=\deg(\mathcal{C}^h)=\deg(\Phi(\mu_n^h,\mu_{n+1}^h)).$ Now, since $\phi_n=\phi_n^h$ and $\phi_{n+1}=\phi_{n+1}^h$, we have $\deg(\mathcal{C}^h)=\deg(\mathcal{C})<\deg\phi_{n+1}=\deg\phi_{n+1}^h=\deg(\mu_{n+1}^h)$, which further implies that $\mathcal{C}^h$ is also essential and \[\deg(\mu_n^h)=\deg(\Phi(\mu_n^h,\mu_{n+1}^h))<\deg(\mu_{n+1}^h).\] Let $\rho_i=[\mu_n;\chi_i,\beta_i].$ Then, for $\mathcal{C}$ and $\mathcal{C}^h$ respectively, we have $\Phi(\mu_n,\mu_{n+1})=[\chi_i]_{\mu_n}$ and $\Phi(\mu_n^h,\mu_{n+1}^h)=[\chi_i^h]_{\mu_n^h}$. As $\phi_n\nin \Phi(\mu_n,\mu_{n+1}),$ \[\phi_n\nsim_{\mu_n}\chi_i,\ \text{i.e.,}\  \mu_n(\phi_n-\chi_i)\leq \mu_n(\phi_n).\] Since each $\chi_i$ is a key polynomial for $\mu_n,$ $\phi_n=\phi_n^h$ implies that $\chi_i=\chi_i^h$ (see Lemma \ref{l1} (i)). Therefore, we have \[\mu_n^h(\phi_n^h-\chi_i^h)\leq \mu_n^h(\phi_n^h),\ \text{i.e.,}\  \phi_n^h\nsim_{\mu_n^h}\chi_i^h,\] and consequently, $\phi_n^h\nin \Phi(\mu_n^h,\mu_{n+1}^h).$ Hence, $\mu_n^h\xrightarrow{\phi_{n+1},\gamma_{n+1}} \mu_{n+1}^h$ is also an MLV step.
 		 \smallskip
 		
 		From the above two cases, (\ref{edf1}) is a chain of mixed augmentations for $\mu^h$, where each augmentation step $\mu_{n}^h\xrightarrow{\phi_{n+1},\gamma_{n+1}} \mu_{n+1}^h$ is an MLV step for every $n$. 
 		
 		If $\mu$ is of Type 1 or 3, then so is $\mu^h$.
 		
 		 Assume now that $\mu$ is of Type 2. In this case, $\mu$ is a stable limit of a continuous family  $\mathcal{C}=(\rho_i)_{i\in\mathbf{A}}$ of augmentations of a valuation $\mu_r$, i.e., $\mu=\rho_{\mathcal{C}}.$ Then, by Lemma \ref{ml1}, $\mu^h$ is also a stable limit of the continuous family $\mathcal{C}^h$ of augmentations of $\mu_r^h$, and $\mu^h=(\rho_\mathcal{C})^h=\rho_{\mathcal{C}^h}$. Since, \[\deg(\mu_r)=\deg(\mathcal{C})=\deg(\Phi(\mu_r,\mu))\  \text{and}\  \phi_r\nin\Phi(\mu_r,\mu),\] arguing as in Case II above, $\phi_r=\phi_r^h$ implies that \[\deg(\mu_r^h)=\deg(\mathcal{C}^h)=\deg(\Phi(\mu_r^h,\mu^h))\ \text{and}\ \phi_r^h\nin\Phi(\mu_r^h,\mu^h),\] it follows that $\mu^h$ is also of Type 2.
 	\end{proof}

In the next result, we characterize the hypothesis, viz., $\phi_n=\phi_n^h$, in terms of the defect and relative gap of each augmentation step of an MLV chain of $\mu.$
	\begin{theorem}\label{mt2}
		Let $\mu$ be any extension of $v$ to $K[x]$ with an MLV chain \begin{align*}
		(v\xrightarrow{\phi_0,\gamma_0})\mu_0\xrightarrow{\phi_1,\gamma_1} \mu_1\xrightarrow{\phi_2,\gamma_2}\cdots \longrightarrow \mu_{n}\xrightarrow{\phi_{n+1},\gamma_{n+1}} \mu_{n+1}\longrightarrow\cdots\longrightarrow \mu. 
	\end{align*}
	Then each $\phi_n$ is irreducible over $K^h$ if and only if the relative gap and defect of each augmentation step $\mu_n\longrightarrow \mu_{n+1}$ are equal.
	\end{theorem}
	\begin{proof}
		We first assume that for each $n$, $\phi_n$ is irreducible over $K^h,$ i.e., $\phi_n=\phi_n^h.$ Consider the augmentation step $\mu_n\longrightarrow \mu_{n+1}.$ The relative gap of this step is  \[d_{n}=\frac{\deg(\mu_{n+1})}{\deg(\Phi(\mu_n,\mu_{n+1}))}= \begin{cases} 1 & \text{if} \ \mu_{n}\longrightarrow \mu_{n+1}\ \text{is ordinary }\\ \deg(\mu_{n+1})/\deg(\mu_{n})\ &  \text{if} \ \mu_{n}\longrightarrow \mu_{n+1} \ \text{is limit }\end{cases}, \] and the defect is  \[d(\mu_n\longrightarrow \mu_{n+1})= \begin{cases} 1 & \text{if} \ \mu_{n}\longrightarrow \mu_{n+1}\ \text{is ordinary }\\ \deg(\mu_{n+1}^h)/\deg(\mu_{n}^h)\ &  \text{if} \ \mu_{n}\longrightarrow \mu_{n+1} \ \text{is limit }\end{cases} \] (see Proposition \ref{dh}). The conclusion is immediate in the ordinary case. 
		
		Assume now that $\mu_n\longrightarrow \mu_{n+1}$ is limit. Since $\phi_n$ is a minimal degree key polynomial for $\mu_n,$	$\phi_n^h$ is a minimal degree key polynomial for $\mu_n^h.$ As $\phi_n=\phi_n^h,$ we have $\deg(\mu_n)=\deg \phi_n=\deg \phi_n^h=\deg(\mu_n^h)$ for every $n.$ Hence, \begin{align}\label{ed}
		\frac{\deg(\mu_{n+1})}{\deg(\mu_n)}=\frac{\deg(\mu_{n+1}^h)}{\deg(\mu_n^h)},	
		\end{align} and the conclusion follows.
		
		Conversely, assume that for each $n$, the relative gap and defect of the augmentation $\mu_n\longrightarrow \mu_{n+1}$ are equal, i.e., (\ref{ed}) holds whenever $\mu_n\longrightarrow \mu_{n+1}$ is a limit augmentation. Since $\deg(\mu_n)=\deg \phi_n$ and $\deg(\mu_n^h)=\deg \phi_n^h$ for every $n,$ (\ref{ed}) implies that
		\begin{align}\label{edf}
			\frac{\deg(\phi_{n+1})}{\deg(\phi_n)}=\frac{\deg(\phi_{n+1}^h)}{\deg(\phi_n^h)}
			\end{align} holds for every limit augmentation.
		 As $\mu_0$ is a depth zero valuation, $\phi_0,$ being its minimal degree key polynomial, has degree one. Therefore, $\phi_0$ is irreducible over $K^h,$ i.e., $\phi_0=\phi_0^h.$ If $\mu_0\longrightarrow \mu_1$ is an ordinary augmentation, then $\phi_1$ is a key polynomial for $\mu_0.$ Therefore, by Lemma \ref{l1} (i), $\phi_0=\phi_0^h$ implies that $\phi_1=\phi_1^h.$ Suppose now that $\mu_0\longrightarrow \mu_1$ is a limit augmentation. Then by (\ref{edf}), for $n=0$, we have
		  \[\frac{\deg\phi_1}{\deg\phi_0}=\frac{\deg\phi_1^h}{\deg\phi_0^h},\
		\text{i.e.,}\ \frac{\deg\phi_1}{\deg\phi_1^h}=\frac{\deg\phi_0}{\deg\phi_0^h}. 
		\]
		Hence,  $\phi_0=\phi_0^h$ implies that $\phi_1=\phi_1^h.$ Proceeding recursively, $\phi_{n}=\phi_{n}^h$ follows from $\phi_{n-1}=\phi_{n-1}^h,$  for each $n\geq1.$ Hence, each $\phi_n$ is irreducible over $K^h.$
		\end{proof}
		
		Note that the above theorem is a generalization of Lemma \ref{nl2} to arbitrary valued fields, for the augmentation steps of an MLV chain of $\mu$.
		\begin{remark}\label{r2}
		With notation as in the above theorem, let $\mu_n\xrightarrow{\phi_{n+1}} \mu_{n+1}$ be an MLV augmentation step. Then, by the definition of defect, for each $Q\in\Phi(\mu_n,\mu_{n+1})$, we have
		\begin{align}\label{equa1}
			d(\mu_n\longrightarrow \mu_{n+1})\leq \deg_{\rho_Q}(\phi_{n+1})\leq \frac{\deg\phi_{n+1}}{\deg Q}=\frac{\deg(\mu_{n+1})}{\deg(\Phi(\mu_n,\mu_{n+1}))}.\end{align} If each $\phi_n$ is irreducible over $K^h,$ then by the above theorem, we have \[d(\mu_n\longrightarrow \mu_{n+1})=\frac{\deg(\mu_{n+1})}{\deg(\Phi(\mu_n,\mu_{n+1}))},\] which in view of (\ref{equa1}) implies that \[\deg_{\rho_Q}(\phi_{n+1})= \frac{\deg\phi_{n+1}}{\deg Q}=m \ (\text{say}).\] Therefore, 
		for each $Q\in\Phi(\mu_n,\mu_{n+1})$, the $Q$-expansion of $\phi_{n+1}$ is of the form \[\phi_{n+1}=q_m{Q}^m+\cdots+q_0,\]  where $q_m=1$ and $\rho_Q(\phi_{n+1})=\rho_Q({Q}^m).$
			\end{remark}
	
		 The  following result follows directly from Theorems \ref{t21} and \ref{t2}.
		\begin{proposition}\label{pc}
		Let $(v\xrightarrow{\phi_0})\mu_0\xrightarrow{\phi_1} \mu_1\xrightarrow{\phi_2}\cdots \longrightarrow \mu_{n}\xrightarrow{\phi_{n+1}} \mu_{n+1}\longrightarrow\cdots\longrightarrow \mu$ be an MLV chain of $\mu.$ Then, the following are equivalent:
		\begin{enumerate}[(i)]
			\item each $\phi_n$ is irreducible over $K^h$,
			\item for each root $\theta_n$ of $\phi_n$, we have \[\overline{w}_{\theta_n,\delta(\phi_n)}|_{K[x]}=\mu_n.\]
		\end{enumerate}
		\end{proposition}
		
		\begin{remark}\label{rc}
		Note that if $\mu_n\longrightarrow \mu_{n+1}$ is a limit augmentation with a continuous family $\mathcal{C}=(\rho_i)_{i\in\mathbf{A}}$, where $\rho_i=[\mu_n;\chi_i,\beta_i],$ then $\Phi(\mu_n,\mu_{n+1})=[\chi_i]_{\mu_n}$. Since the $\chi_i$'s are key polynomials for $\mu_n,$ $\phi_n=\phi_n^h$, it follows that $\chi_i=\chi_i^h$ (see Lemma \ref{l1} (i)). Therefore, by Theorem \ref{t2}, for each root $a$ of $\chi_i$, \[\overline{w}_{a,\delta(\chi_i)}|_{K[x]}=\rho_i,\ \forall \ i\in\mathbf{A}.\]	
		\end{remark} 
		
	We now establish an equivalence between the hypothesis on an MLV chain of $\mu$ and the regularity of a complete sequence of ABKPs for $\mu$. The notion of regularity was introduced by Dutta and Mahboub for a complete sequence of ABKPs in \cite{AM}  as follows.

			\begin{definition}
			Let $\Lambda=\{Q_i\}_{i\in\Delta}$ be a complete sequence of ABKPs for $\mu$ as in Remark \ref{rabkp}. A limit ABKP $Q_{j+1}$ of $\Lambda$, for $j\in I$, is said to be {\bf regular} if there exists $i_0\in\vartheta_{j}$ such that for all $i\geq i_0$, the $Q_i$-expansion of $Q_{j+1}$ has the form
			\[Q_{j+1}=Q_i^{m}+q_{m-1}Q_i^{m-1}+\cdots+q_1Q_1+q_0,\] and
			\[\mu_{Q_i}(Q_{j+1})=m\mu(Q_i)\leq \mu(q_t)+t\mu(Q_i),\ \text{for all}\ t<m.\] If every limit ABKP of $\Lambda$ is regular, then we say that $\Lambda$ is regular.
		\end{definition}
		
		The following result characterizes a regular complete sequence of ABKPs in terms of common extensions of their respective truncated valuations.
		\begin{theorem}(\cite[Theorem 1.3]{AM})\label{am}
			Let $\mu$ be an extension of $v$ to $K[x]$and $\Lambda=\{Q_i\}_{i\in\Delta}$ a complete sequence of ABKPs  for $\mu$. Then the following are equivalent:
			\begin{enumerate}[(i)]
				\item $\Lambda=\{Q_i\}_{i\in\Delta}$ is regular,
				\item for each $i\in\Delta$, and each root $a$ of $Q_i$, we have \[\overline{w}_{a,\delta(Q_i)}|_{K[x]}=\mu_{Q_i}.\]
			\end{enumerate}
		\end{theorem}

		 The connection between a complete sequence of ABKPs and an MLV chain for any valuation-transcendental and valuation-algebraic extension $\mu$ on $K(x)$ is given in \cite{SA1} and \cite{SA2}, respectively. The same holds for a valuation $\mu$ with nontrivial support. These results show that, given a complete sequence of ABKPs, we can explicitly construct an MLV chain of $\mu$, and conversely.
		 
		  For instance, with notation as in Remark \ref{rabkp}, given a complete sequence of ABKPs $\Lambda=\{Q_i\}_{i\in\Delta}$ for $\mu$, an MLV chain of $\mu$ can be constructed as follows: for any $j\in I$, $\mu_{Q_j}=\mu_j$ forms a node of the MLV chain, and the $Q_j$'s are the defining key polynomials. The augmentation $\mu_j\longrightarrow \mu_{j+1}$ is ordinary if and only if $\vartheta_j=\emptyset$. Otherwise, if $Q_{j+1}$ is a limit ABKP, then it is also a limit key polynomial for the continuous family $(\mu_{Q_i})_{i\in\vartheta_j}$ of augmentations of $\mu_j.$ 
		 
		 Keeping the above in mind and in view of Remark \ref{rc}, Proposition \ref{pc} (ii) generalizes Theorem \ref{am} (ii). Hence, the regularity of a complete sequence of ABKPs for $\mu$ is equivalent to the hypothesis on the corresponding MLV chain of $\mu.$  Thus, we say that an MLV chain of a valuation $\mu$ is {\bf regular} if each of its defining key polynomials $\phi_n$ is irreducible over $K^h.$
		 
		 If an MLV chain of $\mu$ is regular, then by Remark \ref{r2}, for any limit ABKP $Q_{j+1}$, the regularity criteria hold not only for all elements of $\vartheta_j$ but  also for any $Q\in\Phi(\mu_j,\mu_{j+1}).$
		
		As an application of Theorem \ref{mt2}, we now show that the notion of regularity is independent of the choice of an MLV chain of $\mu.$
		\begin{corollary}\label{lc}
		The regularity of any MLV chain of a valuation $\mu$ is an invariant of $\mu.$
		\end{corollary}
		\begin{proof}
			Consider two MLV chains of $\mu$  \[(v\xrightarrow{\phi_0})\mu_0\xrightarrow{\phi_1} \mu_1\xrightarrow{\phi_2}\cdots \longrightarrow \mu_{n}\xrightarrow{\phi_{n+1}} \mu_{n+1}\longrightarrow\cdots\longrightarrow \mu,\] 
		\[(v\xrightarrow{\psi_0})w_0\xrightarrow{\psi_1} w_1\xrightarrow{\psi_2}\cdots \longrightarrow w_{n}\xrightarrow{\psi_{n+1}} w_{n+1}\longrightarrow\cdots\longrightarrow \mu,\] and assume that the first one is regular. We need to show that the second chain is also regular. By the uniqueness of MLV chains (see \cite[Theorem 4.7]{EN1}), $w_n\longrightarrow w_{n+1}$ is ordinary (or limit) if and only if $\mu_n\longrightarrow \mu_{n+1}$ is ordinary (or limit), and $\deg(\mu_n)=\deg(w_n)$ for every $n.$ 
		
		If $w_n\longrightarrow w_{n+1}$ is ordinary, then the defect and relative gap of this augmentation step are equal.
		
		 Assume now that $w_n\longrightarrow w_{n+1}$ is a limit augmentation. By Theorem \ref{td1},  \[d(\mu_n\longrightarrow \mu_{n+1})=d(w_n\longrightarrow w_{n+1}),\]  for every $n.$ Also, by assumption, each $\phi_n$ is irreducible over $K^h.$ Therefore, using Theorem \ref{mt2}, in the case of a limit augmentation, we have \[d(\mu_n\longrightarrow \mu_{n+1})=\frac{\deg(\mu_{n+1})}{\deg(\mu_n)}.\] On combining all the above equations, it follows that 
		\[d(w_n\longrightarrow w_{n+1})=d(\mu_n\longrightarrow \mu_{n+1})=\frac{\deg(\mu_{n+1})}{\deg(\mu_n)}=\frac{\deg(w_{n+1})}{\deg(w_n)}.\] Again on using Theorem \ref{mt2}, we have that each $\psi_n$ is irreducible over $K^h,$ and hence, the second chain is also regular.
		\end{proof}

		It now follows that the regularity of any complete sequence of ABKPs for $\mu$ is also an invariant of $\mu.$
		
		In view of the above corollary, we now say that a valuation $\mu$ is {\bf regular} if any MLV chain of $\mu$ is regular. In the next theorem, we show that the regularity of a finite degree valuation $\mu$, depends only on the degree of $\mu.$
		 
		\begin{theorem}\label{lst}
		Let $\mu$ be a valuation on $K[x]$ with finite degree. Then $\mu$ is regular if and only if $\deg(\mu)=\deg(\mu^h)$.	
		\end{theorem}
		\begin{proof}
			The proof is divided into two cases.
			
			\smallskip
			\noindent{\bf Case I:} The valuation $\mu$ is of Type 1 with an MLV chain \[(v\xrightarrow{\phi_0})\mu_0\xrightarrow{\phi_1}\mu_1\xrightarrow{\phi_2}\cdots \longrightarrow \mu_{n}\xrightarrow{\phi_{n+1}} \mu_{n+1}\longrightarrow\cdots\xrightarrow{\phi_r} \mu_r=\mu.\] Then, either $\phi_r$ is a minimal degree key polynomial for $\mu$ or $\text{supp}(\mu)=\phi_r K[x]$, and $\deg(\mu)=\deg\phi_r$. So  we have that either $\phi_r^h \in \operatorname{KP}(\mu^h)$ of minimal degree, or $\phi_r^h K^h[x]=\text{supp}(\mu^h)$ (see Proposition \ref{df1}, Lemmas \ref{df2} and \ref{df3}), hence $\deg (\mu^h)=\deg\phi_r^h.$ 
		
		If $\mu$ is regular, then $\phi_r=\phi_r^h$, and thus, $\deg(\mu)=\deg(\mu^h).$
		
		 For the converse, assume that $\deg(\mu)=\deg(\mu^h).$ Then $\phi_r$ is irreducible over $K^h,$ which in view of Lemma \ref{l1} (ii), implies that $K(\theta)|K$ is unibranched, for every $\theta\in Z(\phi_r).$ Therefore, by (\ref{ev}), we have 
		 \[d(K(\theta)|K)=\prod_{n=0}^{r-1} d(\mu_n\longrightarrow \mu_{n+1})=\prod_{n=0}^{r-1}\frac{\deg(\mu_{n+1})}{\deg(\Phi(\mu_n,\mu_{n+1}))}.\]But in general (see (\ref{equa1})), for each $n,$  $0\leq n<r$,
		 \[d(\mu_n\longrightarrow \mu_{n+1})\leq \frac{\deg(\mu_{n+1})}{\deg(\Phi(\mu_n,\mu_{n+1}))}.\] So, we must have, for each $n$, $0\leq n<r$, \[d(\mu_n\longrightarrow \mu_{n+1})= \frac{\deg(\mu_{n+1})}{\deg(\Phi(\mu_n,\mu_{n+1}))},\]and hence, Theorem \ref{mt2} implies that $\mu$ is regular.
		 
		 \smallskip
		 \noindent{\bf Case II:} The valuation $\mu$ is of Type 2 with an MLV chain \[(v\xrightarrow{\phi_0})\mu_0\xrightarrow{\phi_1}\mu_1\xrightarrow{\phi_2}\cdots \longrightarrow \mu_{n}\xrightarrow{\phi_{n+1}} \mu_{n+1}\longrightarrow\cdots\xrightarrow{\phi_r} \mu_r\xrightarrow{\mathcal{C}} \rho_{\mathcal{C}}=\mu,\]  where $\mu$ is the stable limit of the continuous family $\mathcal{C}=(\rho_i)_{i\in\mathbf{A}}$ of augmentations of $\mu_r$, $\rho_i=[\mu_r;\chi_i,\beta_i],$ and \[\deg(\mu_r)=\deg(\mathcal{C})=\deg(\mu).\]
		 
		 By Lemma \ref{ml1}, $\mu^h$ is also a stable limit of the continuous family $\mathcal{C}^h=(\rho_i^h)_{i\in\mathbf{A}}$ of augmentations of $\mu_r^h$, and \[\deg(\mu_r^h)=\deg(\mathcal{C}^h)=\deg(\mu^h)\] (see Remark \ref{rem2}). 
		 
		 Assume now that $\mu$ is regular, i.e., $\phi_n=\phi_n^h$ for every $n.$ Then, in particular, $\phi_r=\phi_r^h$ implies that $\deg(\mu_r)=\deg(\mu_r^h)$, and hence $\deg(\mu)=\deg(\mu^h).$
		 
		  Conversely, assume that $\deg(\mu)=\deg(\mu^h)$, which implies that $\deg(\mu_r)=\deg(\mu_r^h).$ As $\mu_r$ falls in Case I, $\deg(\mu_r)=\deg(\mu_r^h)$ implies that $\mu_r$ is regular, i.e., each $\phi_n$ is irreducible over $K^h$, for $0\leq n\leq r.$ Hence, $\mu$ is also regular.
		 \end{proof}
	
	\begin{remark}
		From Proposition \ref{df1} and Lemma \ref{l1} (i), if $\mu$ is valuation-transcendental, then $\mu$ is regular if and only if $\phi=\phi^h$ for any $\phi\in \operatorname{KP}(\mu).$
	\end{remark}
	
	Note that, in view of the above theorem, for any finite degree valuation $\mu,$ if $\deg(\mu)=\deg(\mu^h),$ then $\deg(\mu_n)=\deg(\mu_n^h)$ for every node $\mu_n$, $0\leq n\leq r$, in the MLV chain of $\mu$. Similarly, if $\mu$ is regular, then each node $\mu_n$ in the MLV chain of $\mu$ is regular, for $0\leq n\leq r.$ 
	
	 Keeping the above in mind, if $\mu$ has infinite degree, then $\mu$ is regular if and only if each node $\mu_n$ in any MLV chain of $\mu$ is regular.
	
	\begin{corollary}
		If the valuation $\mu$ is regular, then every valuation $\rho$ such that $\rho<\mu$ is also regular.
	\end{corollary}
	\begin{proof}
		Consider an MLV chain of $\mu$ (say) \[(v\xrightarrow{\phi_0})\mu_0\xrightarrow{\phi_1}\mu_1\xrightarrow{\phi_2}\cdots \longrightarrow \mu_{n}\xrightarrow{\phi_{n+1}} \mu_{n+1}\longrightarrow\cdots \mu.\] Since $(-\infty,\mu)_\Gamma$ is totally ordered (see \cite[Theorem 2.4]{EN1}), we have 
		\[(-\infty,\mu)_\Gamma=(-\infty,\mu_0)_\Gamma\cup[\mu_0,\mu_1)_\Gamma\cup\cdots\cup[\mu_n,\mu_{n+1})_\Gamma\cup\cdots.\] 
		
		As $\mu_0$ is depth zero, $\mu_0=w_{\theta_0,\gamma_0}$ with $\phi_0=x-\theta_0$. Therefore, by (\ref{eb}), \[(-\infty,\mu_0)_\Gamma=\{w_{\theta_0,\delta}\mid \delta\in\Gamma,\delta<\gamma_0\},\] and $\phi_0$ is a minimal degree key polynomial for every valuation $\rho\in(-\infty,\mu_0)_\Gamma.$ Since $\phi_0$ being linear, is irreducible over $K^h$, and hence each $\rho\in(-\infty,\mu_0)_\Gamma$ is regular. 
		
		Since $\mu$ is regular, each $\phi_n$ is irreducible over $K^h,$ and by the theorem, each $\mu_n$ is regular. So to prove the result, it is enough to show that every $\rho\in(\mu_n,\mu_{n+1})_\Gamma$ is regular, for each $n.$ 
		
		For an arbitrary but fixed $n,$ $(\mu_n,\mu_{n+1})_\Gamma$ falls into one of the following three cases:
		
			\smallskip
		\noindent{\bf Case I:} The augmentation $\mu_n\longrightarrow \mu_{n+1}$ is ordinary, i.e., $\mu_{n+1}=[\mu_n;\phi_{n+1},\gamma_{n+1}].$ Then, by Lemma 2.7 of \cite{EN1}, we have 
		\[(\mu_n,\mu_{n+1})_\Gamma=\{(\mu_n)_\delta=[\mu_n;\phi_{n+1},\delta]\mid\delta\in\Gamma,\ \mu_n(\phi_{n+1})<\delta<\gamma_{n+1}\},\] i.e., every $\rho\in(\mu_n,\mu_{n+1})_\Gamma$ is an ordinary augmentation of $\mu_n$ defined by $\phi_{n+1}$. Therefore, $\phi_{n+1}$ is a minimal degree key polynomial for $\rho$ and $\deg(\rho)=\deg\phi_{n+1}.$ Since $\phi_{n+1}=\phi_{n+1}^h$, by Lemma \ref{df2}, we have $\deg(\rho)=\deg(\rho^h),$ and hence, by the theorem, $\rho$ is regular.
		
		\smallskip
		\noindent{\bf Case II:}  The valuation $\mu$ is the stable limit of a continuous family $\mathcal{C}=(\rho_i)_{i\in\mathbf{A}}$ of augmentations of $\mu_n,$ with $\rho_i=[\mu_n;\chi_i,\beta_i]$.
		
		 \noindent Then, $\Phi(\mu_n,\mu)=\Phi(\mu_n,\rho_i)=[\chi_i]_{\mu_n},$ i.e., all key polynomials $\chi_i$ are $\mu_n$-equivalent, and the value $\beta_{min}:=\mu_n(\chi_i)$ is independent of $i\in\mathbf{A}$. Let $S$ be the initial segment of $\Gamma_{>\beta_{min}}$ generated by the set $\{\beta_i\mid i\in\mathbf{A}\}.$
		
		 By Lemma 3.8 of \cite{EN1}, we have 
		\[(\mu_n,\mu)_\Gamma=\{\rho_\beta=[\mu_n;\chi_i,\beta]\mid \beta\in S\ \text{and}\ i\in\mathbf{A}\ \text{such that}\ \beta_i>\beta\},\] i.e., every $\rho\in(\mu_n,\mu_{n+1})_\Gamma$ is an ordinary augmentation of $\mu_n$ defined by some $\chi_i$. Therefore, $\chi_i$ is the minimal degree key polynomial for $\rho.$ As $\phi_n$ and $\chi_i$ are key polynomials for $\mu_n,$ by Lemma \ref{l1} (i), $\phi_n=\phi_n^h$ implies that $\chi_i=\chi_i^h$ and hence, $\deg(\rho)=\deg(\rho^h)$. Again by the theorem, $\rho$ is regular.
		
		\smallskip
		\noindent{\bf Case III:} The augmentation $\mu_n\longrightarrow \mu_{n+1}$ is limit, i.e., $\mu_{n+1}=[\mathcal{C};\phi_{n+1},\gamma_{n+1}],$ where $\mathcal{C}=(\rho_i)_{i\in\mathbf{A}}$ is an essential continuous family of augmentations of $\mu_n$, and $\rho_i=[\mu_n;\chi_i,\beta_i]$ for every $i\in\mathbf{A}$.
		
		 \noindent Again, by Lemma 3.8 of \cite{EN1}, we have 
		\[(\mu_n,\mu_{n+1})_\Gamma=\{\rho_\beta\mid \beta\in S\}\cup\{(\mu_n)_\delta=[\mathcal{C};\phi_{n+1},\delta]\mid \delta\in\Gamma, T<\delta<\gamma_{n+1}\},\]
		where $T:=\{\rho_i(\phi_{n+1})\mid i\in\mathbf{A}\},$ and $\rho_\beta$ as in Case II. Therefore, for any $\rho\in(\mu_n,\mu_{n+1})_\Gamma$, either $\rho=\rho_\beta$ or $\rho$ is a limit augmentation of $\mu_n$ defined by $\phi_{n+1}$. If $\rho=\rho_\beta$, then it is regular by Case II. Otherwise, $\phi_{n+1}$ is a minimal degree key polynomial for $\rho$ and $\deg(\rho)=\deg\phi_{n+1}.$ Since $\phi_{n+1}=\phi_{n+1}^h$, this implies that $\deg(\rho)=\deg(\rho^h),$ and hence, by the theorem, $\rho$ is regular.
	\end{proof}
	   
	\section*{Acknowledgement}
	The research of the first author is supported by the UGC
	(Reference no.\ 231620082501).
	
\end{document}